\documentclass[a4paper, intlimits, reqno]{amsart}

\usepackage[english]{babel}
\usepackage[T1]{fontenc}
\usepackage[utf8]{inputenc}

\usepackage{amsmath}
\usepackage{amssymb}
\usepackage{MnSymbol}
\usepackage{amsthm}
\allowdisplaybreaks
\usepackage{amsfonts}
\usepackage{mathrsfs} 
\usepackage{mathtools}
\usepackage{nicefrac}
\usepackage{enumitem}
\usepackage{multicol}
\usepackage{url}
\usepackage{dsfont}
\usepackage[numbers]{natbib}
\usepackage{caption}
\usepackage{tikz}
\usetikzlibrary{patterns,calc,arrows}
\tikzstyle{arrow}=[draw, -latex]
\usepackage{prettyref}

\newrefformat{def}{Definition \ref{#1}}
\newrefformat{rem}{Remark \ref{#1}}
\newrefformat{sect}{Section \ref{#1}}
\newrefformat{sub}{Section \ref{#1}}
\newrefformat{prop}{Proposition \ref{#1}}
\newrefformat{thm}{Theorem \ref{#1}}
\newrefformat{chap}{Chapter \ref{#1}}
\newrefformat{cor}{Corollary \ref{#1}}
\newrefformat{par}{Paragraph \ref{#1}}
\newrefformat{ex}{Example \ref{#1}}
\newrefformat{part}{Part \ref{#1}}
\newrefformat{fig}{Figure \ref{#1}}
\newrefformat{app}{Appendix \ref{#1}}
\newrefformat{que}{Question \ref{#1}}
\newrefformat{cond}{Condition \ref{#1}}
\newrefformat{conv}{Convention \ref{#1}}

\swapnumbers
\newtheoremstyle{dotless}{}{}{\itshape}{}{\bfseries}{}{}{}
\theoremstyle{dotless}
\theoremstyle{plain}
\newtheorem{thm}{Theorem}[section]

\newtheorem{prop}[thm]{Proposition}
\newtheorem{cor}[thm]{Corollary}
\theoremstyle{definition}
\newtheorem{defn}[thm]{Definition}
\newtheorem{rem}[thm]{Remark}

\newtheorem{exa}[thm]{Example}

\newtheorem{cond}[thm]{Condition}

\newcommand{\N} {\mathbb{N}}

\newcommand{\R} {\mathbb{R}}
\newcommand{\C} {\mathbb{C}}

\newcommand{\oacx} {\overline{\operatorname{acx}}}

\DeclareMathOperator{\id}{id}
\DeclareMathOperator{\re}{Re}
\DeclareMathOperator{\im}{Im}

\providecommand{\differential}{\mathrm{d}}
\renewcommand{\d}{\differential}

\newcount\cntloop
\newcommand{\norm}[2]{%
  \cntloop=1 \normrec{#1}{\mskip\thinmuskip#2}}
\newcommand{\normrec}[2]{%
  \ifnum\cntloop<#1
    \advance\cntloop by 1
    \def\next{\normrec{#1}{\left|\!#2\right|\!}}%
  \else
    \def\next{\left|\!#2\right|}%
  \fi
  \next}

\begin{document}

\title[Parameter dependence]{Parameter dependence of solutions of the Cauchy-Riemann equation on spaces of weighted smooth functions}
\author[K.~Kruse]{Karsten Kruse}
\address{TU Hamburg \\ Institut f\"ur Mathematik \\
Am Schwarzenberg-Campus~3 \\
Geb\"aude E \\
21073 Hamburg \\
Germany}
\email{karsten.kruse@tuhh.de}

\subjclass[2010]{Primary 35A01, 35B30, 32W05, 46A63, Secondary 46A32, 46E40}

\keywords{Cauchy-Riemann, parameter dependence, weight, smooth, solvability, vector-valued}

\date{\today}
\begin{abstract}
We study the inhomogeneous Cauchy-Riemann equation on spaces $\mathcal{EV}(\Omega,E)$ of weighted $\mathcal{C}^{\infty}$-smooth 
$E$-valued functions on an open set $\Omega\subset\R^{2}$ whose growth on strips along the real axis 
is determined by a family of continuous weights $\mathcal{V}$ where $E$ is a locally convex 
Hausdorff space over $\C$. 
We derive sufficient conditions on the weights $\mathcal{V}$ such that the kernel $\operatorname{ker}\overline{\partial}$ 
of the Cauchy-Riemann operator $\overline{\partial}$ in $\mathcal{EV}(\Omega):=\mathcal{EV}(\Omega,\C)$ 
has the property $(\Omega)$ of Vogt. 
Then we use previous results and conditions on the surjectivity of the Cauchy-Riemann operator
$\overline{\partial}\colon\mathcal{EV}(\Omega)\to\mathcal{EV}(\Omega)$ 
and the splitting theory of Vogt for Fr\'{e}chet spaces and of 
Bonet and Doma\'nski for (PLS)-spaces to deduce the surjectivity of the Cauchy-Riemann operator on 
the space $\mathcal{EV}(\Omega,E)$ if $E:=F_{b}'$ where $F$ is a Fr\'{e}chet space 
satisfying the condition $(DN)$ or if $E$ is an ultrabornological (PLS)-space having the property $(PA)$. 
As a consequence, for every family of right-hand sides $(f_{\lambda})_{\lambda\in U}$ in $\mathcal{EV}(\Omega)$ 
which depends smoothly, holomorphically or distributionally on a parameter $\lambda$ there is 
a family $(u_{\lambda})_{\lambda\in U}$ in $\mathcal{EV}(\Omega)$ with the same kind of parameter dependence 
which solves the Cauchy-Riemann equation $\overline{\partial}u_{\lambda}=f_{\lambda}$ for all $\lambda\in U$.
\end{abstract}
\maketitle
\section{Introduction}
Let $E$ be a linear space of functions on a set $U$ and 
$P(\partial)\colon\mathcal{F}(\Omega)\to\mathcal{F}(\Omega)$ be a linear partial differential operator
with constant coefficients which acts continuously on a locally convex Hausdorff space of (generalized) differentiable 
scalar-valued functions $\mathcal{F}(\Omega)$ on an open set $\Omega\subset\R^{n}$.
We call the elements of $U$ parameters and say that a family 
$(f_{\lambda})_{\lambda\in U}$ in $\mathcal{F}(\Omega)$ depends on a parameter
w.r.t.\ $E$ if the map $\lambda\mapsto f_{\lambda}(x)$ is an element of $E$ for every $x\in\Omega$. 
The question of parameter dependence is whether for every family 
$(f_{\lambda})_{\lambda\in U}$ in $\mathcal{F}(\Omega)$ depending on a parameter w.r.t.\ $E$ 
there is a family $(u_{\lambda})_{\lambda\in U}$ in $\mathcal{F}(\Omega)$ with the same kind of parameter dependence 
which solves the partial differential equation
\[
P(\partial)u_{\lambda}=f_{\lambda},\quad \lambda\in U.
\]
In particular, it is the question of $\mathcal{C}^{k}$-smooth (holomorphic, distributional, etc.) 
parameter dependence if $E$ is the space $\mathcal{C}^{k}(U)$ of $k$-times continuously partially differentiable functions 
on an open set $U\subset\R^{d}$ (the space $\mathcal{O}(U)$ of holomorphic functions on an open set $U\subset\C$, 
the space of distributions $\mathcal{D}(U)'$ on an open set $U\subset\R^{d}$, etc.). 

The question of parameter dependence has been subject of extensive research varying in the choice of the spaces 
$E$, $\mathcal{F}(\Omega)$ and the properties of the partial differential operator $P(\partial)$, e.g.\ 
being (hypo)elliptic, parabolic or hyperbolic. Even partial differential differential operators $P_{\lambda}(\partial)$ 
where the coefficients also depend $\mathcal{C}^{k}([0,1])$-smoothly \cite{mantlik1990}, 
$\mathcal{C}^{\infty}$-smoothly \cite{treves1962_b, treves1962_a}, 
holomorphically \cite{mantlik1991, mantlik1992, treves1962_b} 
or differentiable resp.\ real analytic \cite{browder1962} on the parameter $\lambda$ were considered. 
The case that the coefficients of the partial differential differential operator $P(x,\partial)$ are non-constant 
functions in $x\in\Omega$ was treated for $\mathcal{F}(\Omega)=\mathscr{A}(\R^{n})$, 
the space of real analytic functions on $\R^{n}$, as well \cite{Bonet1998, bonetdomanski2001}. 

The answer to the question of $\mathcal{C}^{k}$-smooth (holomorphic, distributional, etc.) parameter dependence is 
obviously affirmative if $P(\partial)$ has a linear continuous right inverse. The problem to determine those $P(\partial)$ 
which have such a right inverse was posed by Schwartz in the early 1950s (see \cite[p.\ 680]{farkas2011}).
In the case that $\mathcal{F}(\Omega)$ is the space of $\mathcal{C}^{\infty}$-smooth functions or distributions 
on an open set $\Omega\subset\R^{n}$ the problem was solved in 
\cite{meise_taylor_vogt1989, meise_taylor_vogt1990, meise_taylor_vogt_1994} and in the case of ultradifferentiable functions or 
ultradistributions in \cite{meise_taylor_vogt1996_b} by means of Phragm\'en-Lindel\"of type conditions. 
The case that $\mathcal{F}(\Omega)$ is a space of weighted $\mathcal{C}^{\infty}$-smooth functions on $\Omega=\R^{n}$ 
or its dual was handled in \cite{langenbruch1989, langenbruch1995}, even for some $P(x,\partial)$ with smooth coefficients,
the case of tempered distributions in \cite{langenbruch1990} and of Fourier (ultra-)hyperfunctions 
in \cite{langenbruch2007, langenbruch2009}. For H\"ormander's spaces $B_{p,\kappa}^{loc}(\Omega)$ as $\mathcal{F}(\Omega)$ 
the problem was studied in \cite{v_hermanns2005}. The same problem for differential systems on distributions 
was considered in \cite{domanski_vogt2000} and on ultradifferentiable functions or ultradistributions in 
\cite{mb_hermanns2005}.

The conditions of Phragm\'en-Lindel\"of type were analysed in  
\cite{braun_meise_taylor2000b, 
braun_meise_taylor2000a, 
meise_taylor_vogt1989, 
meise_taylor_vogt1996_a, 
meise_taylor_vogt1997, 
meise_taylor_vogt1998} 
for spaces of $\mathcal{C}^{\infty}$-smooth functions or distributions, 
in 
\cite{braun1995, 
momm1994} 
for spaces of real analytic or ultradifferentiable functions of Roumieu type 
and in 
\cite{braun_meise_taylor2003, 
braun_meise_taylor2006, 
braun_meise_taylor2011}  
for ultradifferentiable functions or ultradistributions of Beurling type.
 
The necessary condition of surjectivity of the partial differential operator $P(\partial)$ was 
studied in many papers, e.g.\ in 
\cite{agranovich1961, frerick_kalmes2010, H2, malgrange1956, wengenroth2011} on $\mathcal{C}^{\infty}$-smooth functions 
and distributions, 
in 
\cite{braun_meise_taylor2001,  
hoermander1973,
langenbruch2000, langenbruch2004a, langenbruch2004b} 
on real analytic functions, 
in \cite{braun1993, cattabriga1981} on Gevrey classes, 
in \cite{braun_meise_vogt1989, braun_meise_vogt1994, langenbruch1996, langenbruch1998, meyer1997}
on ultradifferentiable functions of Roumieu type, 
in \cite{franken_meise1990} on ultradistributions of Beurling type, 
in \cite{bonet_galbis_meise1997, braun_meise_vogt1990} on ultradifferentiable functions and ultradistributions
and in \cite{larcher2014} on the multiplier space $\mathcal{O}_{M}$. 

However, if $P(\partial)\colon\mathcal{C}^{\infty}(\Omega)\to\mathcal{C}^{\infty}(\Omega)$, $\Omega\subset\R^{n}$ open, is elliptic, 
then $P(\partial)$ has a linear right inverse (by means of a Hamel basis of $\mathcal{C}^{\infty}(\Omega)$) 
and it has a continuous right inverse due to Michael's selection theorem \cite[Theorem 3.2'', p.\ 367]{michael1956} 
and \cite[Satz 9.28, p.\ 217]{Kaballo}, 
but $P(\partial)$ has no linear continuous right inverse if $n\geq 2$ by
a result of Grothendieck \cite[Theorem C.1, p.\ 109]{treves1967}. 
Nevertheless, the question of parameter dependence w.r.t.\ $E$ has a positive answer 
for several locally convex Hausdorff spaces $E$ due to 
tensor product techniques. In this case the question of parameter dependence obviously has a positive answer if
the topology of $E$ is stronger than the topology of pointwise convergence on $U$ and 
\[
P(\partial)^{E}\colon\mathcal{C}^{\infty}(\Omega,E)\to\mathcal{C}^{\infty}(\Omega,E)
\] 
is surjective where $\mathcal{C}^{\infty}(\Omega,E)$ is the space of $\mathcal{C}^{\infty}$-smooth $E$-valued functions on $\Omega$ 
and $P(\partial)^{E}$ the version of $P(\partial)$ for $E$-valued functions. 
If $E$ is complete, we have the topological isomorphy $\mathcal{C}^{\infty}(\Omega,E)\cong \mathcal{C}^{\infty}(\Omega)\varepsilon E$ 
where the latter space is Schwartz' $\varepsilon$-product. 
By Grothendieck's classical theory of tensor products \cite{Gro} the $\varepsilon$-product is topologically isomorphic 
to the completion of the projective tensor product $\mathcal{C}^{\infty}(\Omega)\widehat{\otimes}_{\pi} E$, 
implying $\mathcal{C}^{\infty}(\Omega,E)\cong\mathcal{C}^{\infty}(\Omega)\widehat{\otimes}_{\pi} E$, 
since $\mathcal{C}^{\infty}(\Omega)$ with its usual topology is a nuclear space. 
From this tensor product representation and the surjectivity of the elliptic operator $P(\partial)$ 
on the Fr\'echet space $\mathcal{C}^{\infty}(\Omega)$ follows the surjectivity of $P(\partial)^{E}$ by 
\cite[Satz 10.24, p.\ 255]{Kaballo} if $E$ is a Fr\'echet space. 
Hence the answer to the question of $\mathcal{C}^{k}$-smooth or holomorphic parameter dependence is 
affirmative but the case of distributional parameter dependence is not covered as $\mathcal{D}(U)_{b}'$ with the strong 
dual topology is not a Fr\'echet space.
However, the surjectivity result for $P(\partial)^{E}$ can even be extended beyond the class of 
Fr\'echet spaces $E$ due to the splitting theory of Vogt for Fr\'{e}chet spaces \cite{vogt1983, V1} and of 
Bonet and Doma\'nski for (PLS)-spaces \cite{bonetdomanski2006, Dom1}. 
Namely, we have that $P(\partial)^{E}$, $n\geq 2$, is surjective if $E:=F_{b}'$ where $F$ is a Fr\'{e}chet space 
satisfying the condition $(DN)$ by \cite[Theorem 2.6, p.\ 174]{vogt1983} or if $E$ is an ultrabornological 
(PLS)-space having the property $(PA)$ by \cite[Corollary 3.9, p.\ 1112]{D/L} since $\operatorname{ker}P(\partial)$ 
has the property $(\Omega)$ by \cite[Proposition 2.5 (b), p.\ 173]{vogt1983}. 
The latter result covers the case of distributional parameter dependence.

In general, Grothendieck's classical theory of tensor products can be applied if $P(\partial)$ is surjective and
$\mathcal{F}(\Omega)$ is a nuclear Fr\'echet space. If in addition $\operatorname{ker}P(\partial)$ 
has the property $(\Omega)$, the splitting theory of Vogt for Fr\'{e}chet spaces 
and of Bonet and Doma\'nski for (PLS)-spaces can be used. 
In the case that $\mathcal{F}(\Omega)$ is not a Fr\'echet space the question of surjectivity of $P(\partial)^{E}$ 
can still be handled. For (PLS)-spaces $\mathcal{F}(\Omega)$, e.g.\ (ultra-)distributions, 
one can apply the splitting theory of Bonet and Doma\'nski for (PLS)-spaces, and 
for (PLH)-spaces $\mathcal{F}(\Omega)$, e.g.\ $\mathcal{D}_{L^{2}}$ and $B_{2,\kappa}^{loc}(\Omega)$ 
which are non-(PLS)-spaces, the splitting theory of Dierolf and Sieg for (PLH)-spaces \cite{dierolf2014, dierolf_sieg2017} 
is available. 
For applications we refer the reader to the already mentioned papers 
\cite{bonetdomanski2006, Dom1, dierolf2014, dierolf_sieg2017, vogt1983, V1} 
as well as 
\cite{bonet_domanski2007, 
domanski2010, 
domanski2011} 
where $\mathcal{F}(\Omega)$ is the space of ultradistributions of Beurling type 
or of ultradifferentiable functions of Roumieu type and $E$, amongst others, the space of real analytic functions 
and to \cite{kalmes2018} where $\mathcal{F}(\Omega)$ is the space of $\mathcal{C}^{\infty}$-smooth functions or 
distributions. 

Notably, the preceding results imply that the inhomogeneous Cauchy-Riemann equation 
with a right-hand side 
$f\in\mathcal{E}(\Omega,E):=\mathcal{C}^{\infty}(\Omega,E)$, where $\Omega\subset\R^{2}$ is open 
and $E$ a locally convex Hausdorff space over $\C$ whose topology is induced by a system of seminorms
$(p_{\alpha})_{\alpha\in\mathfrak{A}}$, given by 
\begin{equation}\label{eq:CR_non-weighted}
\overline{\partial}^{E}u:=(1/2)(\partial_{1}^{E}+i\partial_{2}^{E})u=f
\end{equation}
has a solution $u\in\mathcal{E}(\Omega,E)$ if $E$ is a Fr\'echet space or $E:=F_{b}'$ where $F$ is a Fr\'{e}chet space 
satisfying the condition $(DN)$ or if $E$ is an ultrabornological (PLS)-space having the property $(PA)$. 
Among these spaces $E$ are several spaces of distributions like $\mathcal{D}(U)'$, the space of tempered distributions, 
the space of ultradistributions of Beurling type etc.
In the present paper we study this problem under the constraint that the 
right-hand side $f$ fulfils additional growth conditions 
given by an increasing family of positive continuous functions 
$\mathcal{V}:=(\nu_{n})_{n\in\N}$ on an increasing sequence of open subsets $(\Omega_{n})_{n\in\N}$ 
of $\Omega$ with $\Omega=\bigcup_{n\in\N}\Omega_{n}$, namely,  
\[
|f|_{n,m,\alpha}:=\sup_{\substack{x\in \Omega_{n}\\ \beta\in\N^{2}_{0},\,|\beta|\leq m}}
p_{\alpha}\bigl((\partial^{\beta})^{E}f(x)\bigr)\nu_{n}(x)<\infty
\]
for every $n\in\N$, $m\in\N_{0}$ and $\alpha\in\mathfrak{A}$. 
Let us call the space of such functions $\mathcal{EV}(\Omega,E)$. Our interest is in 
conditions on $\mathcal{V}$ and $(\Omega_{n})_{n\in\N}$ such that there is a 
solution $u\in\mathcal{EV}(\Omega,E)$ of \eqref{eq:CR_non-weighted}, i.e.\ 
we search for conditions that guarantee the surjectivity of 
\[
 \overline{\partial}^{E}\colon \mathcal{EV}(\Omega,E)\to \mathcal{EV}(\Omega,E).
\]
From the previous considerations for the Cauchy-Riemann operator on the space of non-weighted 
$\mathcal{C}^{\infty}$-smooth functions our task is evident and a part of it is already done. 
The spaces $\mathcal{EV}(\Omega):=\mathcal{EV}(\Omega,\C)$ are Fr\'echet spaces by \cite[3.4 Proposition, p.\ 6]{kruse2018_2}, 
in \cite[3.1 Theorem, p.\ 12]{kruse2018_4} we derived conditions 
on the family of weights $\mathcal{V}$ and the sequence of sets $(\Omega_{n})_{n\in\N}$ 
such that $\mathcal{EV}(\Omega)$ becomes a nuclear space and in \cite[4.8 Theorem, p.\ 20]{kruse2018_5} such that 
$\overline{\partial}$ is surjective on $\mathcal{EV}(\Omega)$. 
Furthermore, we obtained the topological isomorphy $\mathcal{EV}(\Omega,E)\cong\mathcal{EV}(\Omega)\varepsilon E$ 
for complete $E$ in \cite[5.10 Example c), p.\ 24]{kruse2017}.
Therefore we already have a solution in the case that $E$ is Fr\'echet space at hand 
(see \cite[4.9 Corollary, p.\ 21]{kruse2018_5}). 
What remains to be done is to characterise conditions 
on the kernel $\operatorname{ker}\overline{\partial}$ in $\mathcal{EV}(\Omega)$
to have the property $(\Omega)$ which allow us to extend the surjectivity result beyond 
the class of Fr\'echet spaces $E$. Concerning the sequence $(\Omega_{n})_{n\in\N}$, 
we concentrate on the case that it is a sequence of strips along the real axis, i.e.\ 
$\Omega_{n}:=\{z\in\C\;|\;|\im(z)|<n\}$. 
The case that this sequence has holes along the real axis is treated in \cite{kruse2019_2}.

Let us briefly outline the content of our paper. In Section 2 we summarise the necessary definitions 
and preliminaries which are needed in the subsequent sections. 
The kernel $\operatorname{ker}\overline{\partial}$ is a projective limit and in Section 3 
we prove that it is weakly reduced under suitable assumptions on $\mathcal{V}$ and $(\Omega_{n})_{n\in\N}$ 
(see \prettyref{cor:weakly_reduced}). The weak reducibility is used in Section 4 to obtain 
property $(\Omega)$ for the kernel in the case that $(\Omega_{n})_{n\in\N}$ is a sequence of strips
along the real axis (see \prettyref{thm:omega_for_strips}, \prettyref{cor:omega_for_strips}). 
In our final Section 5 we use the preceding conditions on the weights $\mathcal{V}$ 
to deduce the surjectivity of the Cauchy-Riemann operator on $\mathcal{EV}(\Omega,E)$ 
for $E:=F_{b}'$ where $F$ is a Fr\'{e}chet space 
satisfying the condition $(DN)$ or an ultrabornological (PLS)-space $E$ having the property $(PA)$ 
(see \prettyref{thm:surj_CR_DN_PA}). In particular, we apply our results in the case that 
$(\Omega_{n})_{n\in\N}$ is a sequence of strips along the real axis
(see \prettyref{cor:surj_CR_DN_PA}) and 
for example $\nu_{n}(z):=\exp(a_{n}|\re(z)|^{\gamma})$ for some $0<\gamma\leq 1$ and $a_{n}\nearrow 0$ 
(see \prettyref{cor:surj_CR_DN_PA_exa}).
\section{Notation and Preliminaries}
The notation and preliminaries are essentially the same as in \cite[Section 2]{kruse2017, kruse2018_5}.
We define the distance of two subsets $M_{0}, M_{1} \subset\R^{2}$ w.r.t.\ a norm $\|\cdot\|$ on $\R^{2}$ via
\[
  \d^{\|\cdot\|}(M_{0},M_{1}) 
:=\begin{cases}
   \inf_{x\in M_{0},\,y\in M_{1}}\|x-y\| &,\;  M_{0},\,M_{1} \neq \emptyset, \\
   \infty &,\;  M_{0}= \emptyset \;\text{or}\; M_{1}=\emptyset.
  \end{cases}
\]
Moreover, we denote by $\|\cdot\|_{\infty}$ the sup-norm, by $|\cdot|$ the Euclidean norm on $\R^{2}$, by
$\mathbb{B}_{r}(x):=\{w\in\R^{2}\;|\;|w-x|<r\}$ the Euclidean ball around $x\in\R^{2}$ with radius $r>0$ and 
identify $\R^{2}$ and $\C$ as (normed) vector spaces.
We denote the complement of a subset $M\subset \R^{2}$ by $M^{C}:= \R^{2}\setminus M$, 
the closure of $M$ by $\overline{M}$ and the boundary of $M$ by $\partial M$.
For a function $f\colon M\to\C$ and $K\subset M$ we denote by $f_{\mid K}$ the restriction of $f$ to $K$ and by 
\[
 \|f\|_{K}:=\sup_{x\in K}|f(x)|
\]
the sup-norm on $K$. By $L^{1}(\Omega)$ we denote the space of (equivalence classes of) 
$\C$-valued Lebesgue integrable functions on a measurable set $\Omega\subset\R^{2}$ 
and by $L^{q}(\Omega)$, $q\in\N$, the space of functions $f$ such that $f^{q}\in L^{1}(\Omega)$. 

By $E$ we always denote a non-trivial locally convex Hausdorff space over the field 
$\C$ equipped with a directed fundamental system of seminorms $(p_{\alpha})_{\alpha\in \mathfrak{A}}$. 
If $E=\C$, then we set $(p_{\alpha})_{\alpha\in \mathfrak{A}}:=\{|\cdot|\}$.
Further, we denote by $L(F,E)$ the space of continuous linear maps from 
a locally convex Hausdorff space $F$ to $E$ and sometimes write $\langle T,f\rangle:=T(f)$, $f\in F$, for $T\in L(F,E)$. 
If $E=\C$, we write $F':=L(F,\C)$ for the dual space of $F$. 
If $F$ and $E$ are (linearly topologically) isomorphic, we write $F\cong E$.
We denote by $L_{t}(F,E)$ the 
space $L(F,E)$ equipped with the locally convex topology of uniform convergence 
on the finite subsets of $F$ if $t=\sigma$, on the precompact subsets of $F$ if $t=\gamma$, 
on the absolutely convex, compact subsets of $F$ if $t=\kappa$ and on the bounded subsets of $F$ if $t=b$.\\
The so-called $\varepsilon$-product of Schwartz is defined by 
\begin{equation}\label{notation0}
F\varepsilon E:=L_{e}(F_{\kappa}',E)
\end{equation}
where $L(F_{\kappa}',E)$ is equipped with the topology of uniform convergence on equicontinuous subsets of $F'$. 
This definition of the $\varepsilon$-product coincides with the original 
one by Schwartz \cite[Chap.\ I, \S1, D\'{e}finition, p.\ 18]{Sch1}. 

We recall the following well-known definitions concerning continuous partial differentiability of 
vector-valued functions (c.f.\ \cite[p.\ 4]{kruse2018_2}). A function $f\colon\Omega\to E$ on an open set 
$\Omega\subset\mathbb{R}^{2}$ to $E$ is called continuously partially differentiable ($f$ is $\mathcal{C}^{1}$) 
if for the $n$-th unit vector $e_{n}\in\mathbb{R}^{2}$ the limit
\[
(\partial^{e_{n}})^{E}f(x):=(\partial_{n})^{E}f(x)
:=\lim_{\substack{h\to 0\\ h\in\mathbb{R}, h\neq 0}}\frac{f(x+he_{n})-f(x)}{h}
\]
exists in $E$ for every $x\in\Omega$ and $(\partial^{e_{n}})^{E}f$ 
is continuous on $\Omega$ ($(\partial^{e_{n}})^{E}f$ is $\mathcal{C}^{0}$) for every $n\in\{1,2\}$. 
For $k\in\mathbb{N}$ a function $f$ is said to be $k$-times continuously partially differentiable 
($f$ is $\mathcal{C}^{k}$) if $f$ is $\mathcal{C}^{1}$ and all its first partial derivatives are $\mathcal{C}^{k-1}$.
A function $f$ is called infinitely continuously partially differentiable ($f$ is $\mathcal{C}^{\infty}$) 
if $f$ is $\mathcal{C}^{k}$ for every $k\in\mathbb{N}$.
The linear space of all functions $f\colon\Omega\to E$ which are $\mathcal{C}^{\infty}$ 
is denoted by $\mathcal{C}^{\infty}(\Omega,E)$. 
Let $f\in\mathcal{C}^{\infty}(\Omega,E)$. For $\beta=(\beta_{n})\in\mathbb{N}_{0}^{2}$ we set 
$(\partial^{\beta_{n}})^{E}f:=f$ if $\beta_{n}=0$, and
\[
(\partial^{\beta_{n}})^{E}f
:=\underbrace{(\partial^{e_{n}})^{E}\cdots(\partial^{e_{n}})^{E}}_{\beta_{n}\text{-times}}f
\]
if $\beta_{n}\neq 0$ as well as 
\[
(\partial^{\beta})^{E}f
:=(\partial^{\beta_{1}})^{E}(\partial^{\beta_{2}})^{E}f.
\]
Due to the vector-valued version of Schwarz' theorem $(\partial^{\beta})^{E}f$ is independent of the order of the partial 
derivatives on the right-hand side, we call $|\beta|:=\beta_{1}+\beta_{2}$ the order of differentiation 
and write $\partial^{\beta}f:=(\partial^{\beta})^{\C}f$. 

A function $f\colon\Omega\to E$ on an open set 
$\Omega\subset\mathbb{C}$ to $E$ is called holomorphic if the limit
\[
\bigl(\frac{\partial}{\partial z}\bigr)^{E}f(z_{0})
:=\lim_{\substack{h\to 0\\ h\in\mathbb{C}, h\neq 0}}\frac{f(z_{0}+h)-f(z_{0})}{h}
\]
exists in $E$ for every $z_{0}\in\Omega$ and the space of such functions is denoted by $\mathcal{O}(\Omega,E)$. 
The exact definition of the spaces from the introduction is as follows. 

\begin{defn}[{\cite[3.1 Definition, p.\ 5]{kruse2018_2}}]\label{def:smooth_weighted_space}
Let $\Omega\subset\R^{2}$ be open and $(\Omega_{n})_{n\in\N}$ a family of non-empty
open sets such that $\Omega_{n}\subset\Omega_{n+1}$ and $\Omega=\bigcup_{n\in\N} \Omega_{n}$.
Let $\mathcal{V}:=(\nu_{n})_{n\in\N}$ be a countable family of positive continuous functions 
$\nu_{n}\colon \Omega \to (0,\infty)$ such that $\nu_{n}\leq\nu_{n+1}$ for all $n\in\N$.
We call $\mathcal{V}$ a directed family of continuous weights on $\Omega$ and set for $n\in\N$ 
\begin{enumerate}
 \item [a)]
 \[
\mathcal{E}\nu_{n}(\Omega_{n}, E):= \{ f \in \mathcal{C}^{\infty}(\Omega_{n}, E)\; | \;
\forall\;\alpha\in\mathfrak{A},\,m \in \N_{0}^{2}:\; |f|_{n,m,\alpha} < \infty \}
\]
and 
\[
\mathcal{EV}(\Omega, E):=\{ f\in \mathcal{C}^{\infty}(\Omega, E)\; | \;\forall\; n \in \N:
\; f_{\mid\Omega_{n}}\in \mathcal{E}\nu_{n}(\Omega_{n}, E)\}
\]
where
\[
|f|_{n,m,\alpha}:=\sup_{\substack{x \in \Omega_{n}\\ \beta \in \N_{0}^{2}, \, |\beta| \leq m}}
p_{\alpha}\bigl((\partial^{\beta})^{E}f(x)\bigr)\nu_{n}(x).
\]
\item [b)] 
\[
\mathcal{E}\nu_{n,\overline{\partial}}(\Omega_{n},E):= \{ f \in \mathcal{E}\nu_{n}(\Omega_{n}, E)\; | \;
f\in\operatorname{ker}\overline{\partial}^{E} \}
\]
and 
\[
\mathcal{EV}_{\overline{\partial}}(\Omega, E):=\{f\in\mathcal{EV}(\Omega, E)\;|\;
f\in\operatorname{ker}\overline{\partial}^{E}\}.
\]
\item [c)] 
\[
\mathcal{O}\nu_{n}(\Omega_{n},E):= \{ f \in \mathcal{O}(\Omega_{n}, E)\; | \;
\forall\;\alpha\in\mathfrak{A}:\;|f|_{n,\alpha} < \infty  \}
\]
and 
\[
\mathcal{OV}(\Omega, E):=\{ f\in \mathcal{O}(\Omega, E)\; | \;\forall\; n \in \N:
\; f_{\mid\Omega_{n}}\in \mathcal{O}\nu_{n}(\Omega_{n}, E)\}
\]
where 
\[
|f|_{n,\alpha}:=\sup_{\substack{x \in \Omega_{n}}}
p_{\alpha}(f(x))\nu_{n}(x).
\]
\end{enumerate}
The subscript $\alpha$ in the notation of the seminorms is omitted in the $\C$-valued case. 
The letter $E$ is omitted in the case $E=\C$ as well, e.g.\ we write
$\mathcal{E}\nu_{n}(\Omega_{n}):=\mathcal{E}\nu_{n}(\Omega_{n},\C)$ 
and $\mathcal{EV}(\Omega):=\mathcal{EV}(\Omega,\C)$ .
\end{defn}

The spaces $\mathcal{FV}(\Omega,E)$, $\mathcal{F}=\mathcal{E}$, $\mathcal{O}$, are projective limits, namely, we have
\[
\mathcal{FV}(\Omega, E)\cong \lim_{\substack{\longleftarrow\\n\in \N}}\mathcal{F}\nu_{n}(\Omega_{n}, E)
\]
where the spectral maps are given by the restrictions
\[
\pi_{k,n}\colon \mathcal{F}\nu_{k}(\Omega_{k}, E)\to \mathcal{F}\nu_{n}(\Omega_{n}, E),\;
f\mapsto f_{\mid\Omega_{n}},\;k\geq n.
\]

\section{Weak reducibility of $\mathcal{OV}(\Omega)$}
The goal of this section is to show that the projective limit $\mathcal{OV}(\Omega)$ is weakly reduced 
under suitable assumptions, i.e.\ for every $n\in\N$ there is $m\in\N$ such that 
$\mathcal{OV}(\Omega)$ is dense in $\mathcal{O}\nu_{m}(\Omega_{m})$
w.r.t.\ the topology of $\mathcal{O}\nu_{n}(\Omega_{n})$. First, we show that $\mathcal{OV}(\Omega)$ and 
$\mathcal{EV}_{\overline{\partial}}(\Omega)$ coincide topologically under mild assumptions 
on weights $\mathcal{V}$ and the sequence of sets $(\Omega_{n})$. 
Then we use a similar result for $\mathcal{EV}_{\overline{\partial}}(\Omega)$ which was obtained in \cite{kruse2018_5} 
to prove the weak reducibility of $\mathcal{OV}(\Omega)$. For corresponding results in the case that 
$\Omega_{n}=\Omega$ for all $n\in\N$ see \cite[Theorem 3, p.\ 56]{Epifanov1992}, \cite[1.3 Lemma, p.\ 418]{Langenbruch1994} 
and \cite[Theorem 1, p.\ 145]{Polyakova2017}.

\begin{cond}[{\cite[3.3 Condition, p.\ 7]{kruse2018_5}}]\label{cond:weights}
Let $\mathcal{V}:=(\nu_{n})_{n\in\N}$ be a directed family of continuous weights on an open set $\Omega\subset\R^{2}$ 
and $(\Omega_{n})_{n\in\N}$ a family of non-empty open sets such that $\Omega_{n}\subset\Omega_{n+1}$
and $\Omega=\bigcup_{n\in\N} \Omega_{n}$.
For every $k\in\N$ let there be $\rho_{k}\in \R$ such that $0<\rho_{k}<\d^{\|\cdot\|_{\infty}}(\{x\},\partial\Omega_{k+1})$
for all $x\in\Omega_{k}$ and let there be $q\in\N$ such that for any $n\in\N$ there is 
$\psi_{n}\in L^{q}(\Omega_{k})$, $\psi_{n}>0$, and $J_{i}(n)\geq n$ and $C_{i}(n)>0$ such that for any $x\in\Omega_{k}$:
\begin{enumerate}
  \item [$(\omega.1)\phantom{^{q}}$] $\sup_{\zeta\in\R^{2},\,\|\zeta\|_{\infty}\leq \rho_{k}}\nu_{n}(x+\zeta)
  \leq C_{1}(n)\inf_{\zeta\in\R^{2},\,\|\zeta\|_{\infty}\leq \rho_{k}}\nu_{J_{1}(n)}(x+\zeta)$
  \item [$(\omega.2)^{q}$] $\nu_{n}(x)\leq C_{2}(n)\psi_{n}(x)\nu_{J_{2}(n)}(x)$
\end{enumerate}
\end{cond}

\begin{exa}[{\cite[3.7 Example, p.\ 9]{kruse2018_5}}]\label{ex:weights_little_omega}
Let $\Omega\subset\R^{2}$ be open and $(\Omega_{n})_{n\in\N}$ a family of non-empty open sets such that
\begin{enumerate}
 \item [(i)] $\Omega_{n}:=\R^{2}$ for every $n\in\N$.
 \item [(ii)] $\Omega_{n}\subset\Omega_{n+1}$ and $\d^{\|\cdot\|}(\Omega_{n},\partial\Omega_{n+1})>0$ for every $n\in\N$.
 \item [(iii)] $\Omega_{n}:=\{x=(x_{i})\in \Omega\;|\;\forall\;i\in I:\;|x_{i}|<n+N\;\text{and}\;
 \d^{\|\cdot\|}(\{x\},\partial \Omega)>1/(n+N) \}$ where $I\subset\{1,2\}$, $\partial \Omega\neq\varnothing$ and $N\in\N_{0}$ 
 is big enough.
 \item [(iv)] $\Omega_{n}:=\{x=(x_{i})\in \Omega\;|\;\forall\;i\in I:\;|x_{i}|<n\}$ where $I\subset\{1,2\}$ and $\Omega:=\R^{2}$.
 \item [(v)] $\Omega_{n}:=\mathring K_{n}$ where $K_{n}\subset\mathring K_{n+1}$, $\mathring{K}_{n}\neq\varnothing$, 
 is a compact exhaustion of $\Omega$.
\end{enumerate}
Let $(a_{n})_{n\in\N}$ be strictly increasing such that $a_{n}\geq 0$ for all $n\in\N$ or 
$a_{n}\leq 0$ for all $n\in\N$. The family $\mathcal{V}:=(\nu_{n})_{n\in\N}$ of positive continuous functions on $\Omega$ given by 
\[
 \nu_{n}\colon\Omega\to (0,\infty),\;\nu_{n}(x):=e^{a_{n}\mu(x)},
\]
with some function $\mu\colon\Omega\to[0,\infty)$ fulfils $\nu_{n}\leq\nu_{n+1}$ for all $n\in\N$ and
\prettyref{cond:weights} for every $q\in\N$ with $\psi_{n}(x):= (1+|x|^{2})^{-2}$, $x\in\R^{2}$, for every $n\in\N$ if
\begin{enumerate}
  \item [a)] there is some $0<\gamma\leq 1$ such that
  $\mu(x)=|(x_{i})_{i\in I_{0}}|^{\gamma}$, $x=(x_{1},x_{2})\in\Omega$, where $I_{0}:=\{1,2\}\setminus I$ with $I\subsetneq\{1,2\}$ and 
  $(\Omega_{n})_{n\in\N}$ from (iii) or (iv).
  \item [b)] $\lim_{n\to\infty}a_{n}=\infty$ or $\lim_{n\to\infty}a_{n}=0$ and there is some 
  $m\in\N$, $m\leq 5$, such that $\mu(x)=|x|^{m}$, $x\in\Omega$, with $(\Omega_{n})_{n\in\N}$ from (i) or (ii).
  \item [c)] $a_{n}=n/2$ for all $n\in\N$ and $\mu(x)=\ln(1+|x|^{2})$, $x\in\R^{2}$, with $(\Omega_{n})_{n\in\N}$ from (i).
  \item [d)] $\mu(x)=0$, $x\in\Omega$, with $(\Omega_{n})_{n\in\N}$ from (v).
\end{enumerate}
\end{exa}

In this section we only need property $(\omega.1)$. 
 
\begin{prop}\label{prop:equiv_seminorms}
Let $\mathcal{V}:=(\nu_{n})_{n\in\N}$ be a directed family of continuous weights on an open set $\Omega\subset\R^{2}$ 
and $(\Omega_{n})_{n\in\N}$ a family of non-empty open sets such that $\Omega_{n}\subset\Omega_{n+1}$
and $\Omega=\bigcup_{n\in\N} \Omega_{n}$. If $(\omega.1)$ is fulfilled, then
\begin{enumerate}
\item [a)] for every $n\in\N$ and $m\in\N_{0}$ there is $C>0$ such that 
\[
|f|_{n,m}\leq C|f|_{2J_{1}(n)},\quad f\in \mathcal{O}\nu_{2J_{1}(n)}(\Omega_{2J_{1}(n)}).
\]
\item [b)] $\mathcal{EV}_{\overline{\partial}}(\Omega)=\mathcal{OV}(\Omega)$ as Fr\'echet spaces.
\end{enumerate}
\end{prop}
\begin{proof}
$a)$ Let $n\in\N$ and $m\in\N_{0}$. First, we note that $\Omega_{n+1}\subset\Omega_{2J_{1}(n)}$ and 
$\partial^{\beta}f(x)=i^{\beta_{2}}f^{(|\beta|)}(x)$, $x\in\Omega_{2J_{1}(n)}$, holds 
for all $\beta=(\beta_{1},\beta_{2})\in\N_{0}^{2}$ and $f\in\mathcal{O}\nu_{2J_{1}(n)}(\Omega_{2J_{1}(n)})$ 
where $f^{(|\beta|)}$ is the $|\beta|$th complex derivative of $f$.
Then we obtain via $(\omega.1)$ and Cauchy's inequality 
\begin{align*}
|f|_{n,m}&=\sup_{\substack{x \in \Omega_{n}\\ \beta \in \N_{0}^{2}, \, |\beta| \leq m}}
|\partial^{\beta}f(x)|\nu_{n}(x)
\leq \sup_{\substack{x \in \Omega_{n}\\ \beta \in \N_{0}^{2}, \, |\beta| \leq m}}
\frac{|\beta|!}{\rho_{n}^{|\beta|}}\max_{\substack{\zeta\in\R^{2}\\|\zeta-x|=\rho_{n}}}|f(\zeta)|\nu_{n}(x)\\
&\underset{\mathclap{(\omega.1)}}{\leq} C_{1}\sup_{\substack{x \in \Omega_{n}\\ \beta \in \N_{0}^{2}, \, |\beta| \leq m}}
\frac{|\beta|!}{\rho_{n}^{|\beta|}}\max_{\substack{\zeta\in\R^{2}\\|\zeta-x|=\rho_{n}}}|f(\zeta)|\nu_{J_{1}(n)}(\zeta)\\
&\leq C_{1}\sup_{\beta \in \N_{0}^{2}, \, |\beta| \leq m}
\frac{|\beta|!}{\rho_{n}^{|\beta|}}\sup_{\zeta \in \Omega_{n+1}}|f(\zeta)|\nu_{J_{1}(n)}(\zeta)
\leq C_{1}\sup_{\beta \in \N_{0}^{2}, \, |\beta| \leq m}\frac{|\beta|!}{\rho_{n}^{|\beta|}}|f|_{2J_{1}(n)}.
\end{align*}

$b)$ The space $\mathcal{EV}_{\overline{\partial}}(\Omega)$ is a Fr\'echet space since it is a closed subspace of the 
Fr\'echet space $\mathcal{EV}(\Omega)$ by \cite[3.4 Proposition, p.\ 6]{kruse2018_2}. 
From part a) and $|f|_{n}=|f|_{n,0}$ for all $n\in\N$ and 
$f\in\mathcal{EV}_{\overline{\partial}}(\Omega)$ follows the statement.
\end{proof}

If in addition $(\omega.2)^{1}$ is fulfilled, then the space $\mathcal{EV}(\Omega)$ is nuclear
and thus its subspace $\mathcal{OV}(\Omega)$ as well which we need in our last section.
The following conditions guarantee a kind of weak reducibility of the projective limit $\mathcal{EV}(\Omega)$.

\begin{cond}[{\cite[4.2 Condition, p.\ 10]{kruse2018_5}}]\label{cond:dense}
Let $\mathcal{V}:=(\nu_{n})_{n\in\N}$ be a directed family of continuous weights on an open set
$\Omega\subset\R^{2}$ and $(\Omega_{n})_{n\in\N}$ a family of non-empty
open sets such that $\Omega_{n}\neq \R^{2}$, $\Omega_{n}\subset\Omega_{n+1}$ for all $n\in\N$, 
$\d_{n,k}:=\d^{|\cdot|}(\Omega_{n},\partial\Omega_{k})>0$ for all $n,k\in\N$, $k>n$,
and $\Omega=\bigcup_{n\in\N} \Omega_{n}$.\\
a) For every $n\in\N$ let there be $g_{n}\in\mathcal{O}(\C)$ with $g_{n}(0)=1$ and $I_{j}(n)> n$ such that
\begin{enumerate}
\item [(i)] for every $\varepsilon>0$ there is a compact set $K\subset \overline{\Omega}_{n}$ with 
$\nu_{n}(x)\leq\varepsilon\nu_{I_{1}(n)}(x)$ for all $x\in\Omega_{n}\setminus K$.
\item [(ii)] there is an open set $X_{I_{2}(n)}\subset\R^{2}\setminus \overline{\Omega}_{I_{2}(n)}$ such that
there are $R_{n},r_{n}\in\R$ with $0<2R_{n}<\d^{|\cdot|}(X_{I_{2}(n)},\Omega_{I_{2}(n)}):=\d_{X,I_{2}(n)}$ 
and $R_{n}<r_{n}<\d_{X,I_{2}(n)}-R_{n}$ as well as 
$A_{2}(\cdot,n)\colon X_{I_{2}(n)}+\mathbb{B}_{R_{n}}(0)\to (0,\infty)$, 
$A_{2}(\cdot,n)_{\mid X_{I_{2}(n)}}$ locally bounded, satisfying
\begin{equation}\label{pro.2}
\max\{|g_{n}(\zeta)|\nu_{I_{2}(n)}(z)\;|\;\zeta\in\R^{2},\,|\zeta-(z-x)|=r_{n}\}\leq A_{2}(x,n) 
\end{equation}
for all $z\in\Omega_{I_{2}(n)}$ and $x\in X_{I_{2}(n)}+\mathbb{B}_{R_{n}}(0)$.
\item [(iii)] for every compact set $K\subset \R^{2}$ there is $A_{3}(n,K)>0$  with
\[
\int_{K}{\frac{|g_{n}(x-y)|\nu_{n}(x)}{|x-y|}\d y}\leq A_{3}(n,K),\quad x\in \Omega_{n}.
\]
\end{enumerate}
b) Let a)(i) be fulfilled. For every $n\in\N$ let there be $I_{4}(n)>n$ and $A_{4}(n)>0$ such that
\begin{equation}\label{pro.4}
\int_{\Omega_{I_{4}(n)}}{\frac{|g_{I_{14}(n)}(x-y)|\nu_{p}(x)}{|x-y|\nu_{k}(y)}\d y}\leq A_{4}(n), \quad x\in \Omega_{p},
\end{equation}
for $(k,p)=(I_{4}(n),n)$ and $(k,p)=(I_{14}(n),I_{14}(n))$ where $I_{14}(n):=I_{1}(I_{4}(n))$.\\
c) Let a)(i)-(ii) and b) be fulfilled. For every $n\in\N$, every closed subset $M\subset \overline{\Omega}_{n}$ 
and every component $N$ of $M^{C}$ we have
\[
N\cap \overline{\Omega}_{n}^{C}\neq \varnothing\;\Rightarrow\; N\cap X_{I_{214}(n)}\neq \varnothing
\]
where $I_{214}(n):=I_{2}(I_{14}(n))$.
\end{cond}

We will see that $\Omega_{n}:=\{z\in\C\;|\;|\im(z)|<n\}$ and $\nu_{n}(z):=\exp(a_{n}|\re(z)|^{\gamma})$ for some $0<\gamma\leq 1$ and $a_{n}\nearrow 0$ or $a_{n}\nearrow \infty$ 
fulfil the conditions above with $g_{n}(z):=\exp(-z^2)$.

\begin{thm}[{\cite[4.3 Theorem, p.\ 10]{kruse2018_5}}]\label{thm:dense_proj_lim}
Let $n\in\N$. If \prettyref{cond:dense} is fulfilled, then
$\pi_{I_{214}(n),n}(\mathcal{E}\nu_{I_{214}(n),\overline{\partial}}(\Omega_{I_{214}(n)}))$ is 
dense in $\pi_{I_{14}(n),n}(\mathcal{E}\nu_{I_{14}(n),\overline{\partial}}(\Omega_{I_{14}(n)}))$ w.r.t.\
$(|\cdot|_{n,m})_{m\in\N_{0}}$ .
\end{thm}

As a consequence of this theorem we obtain that the projective limit $\mathcal{OV}(\Omega)$ is weakly reduced 
which is a generalisation of \cite[5.6 Corollary, p.\ 69]{ich} and \cite[5.11 Corollary, p.\ 75]{ich}.

\begin{cor}\label{cor:weakly_reduced}
Let $n\in\N$. If \prettyref{cond:dense} with $I_{214}(k)\geq I_{14}(k+1)$ for all $k\in\N$ and $(\omega.1)$ hold, then 
$\pi_{n}(\mathcal{OV}(\Omega))$ is dense in $\pi_{n,2J_{1}I_{14}(n)}(\mathcal{O}\nu_{2J_{1}I_{14}(n)}(\Omega_{2J_{1}I_{14}(n)}))$ 
w.r.t.\ $|\cdot|_{n}$ where $J_{1}I_{14}(n):=J_{1}(I_{14}(n))$ and
\[
\pi_{n}\colon\mathcal{OV}(\Omega)\rightarrow\mathcal{O}\nu_{n}(\Omega_{n}),\;\pi_{n}(f):=f_{\mid\Omega_{n}}.
\]
\end{cor}
\begin{proof}
We omit the restriction maps in our proof. Due to \prettyref{prop:equiv_seminorms} a) 
the restrictions to $\Omega_{I_{14}(n)}$ of functions from
$\mathcal{O}\nu_{2J_{1}I_{14}(n)}(\Omega_{2J_{1}I_{14}(n)})$ are elements of
$\mathcal{E}\nu_{I_{14}(n),\overline{\partial}}(\Omega_{I_{14}(n)})$. 
Let $\varepsilon>0$ and $f_{0}\in\mathcal{O}\nu_{2J_{1}I_{14}(n)}(\Omega_{2J_{1}I_{14}(n)})$. 
For every $j\in\mathbb{N}$ there exists 
\begin{enumerate}
\item [(i)]  $f_{j}\in\mathcal{E}\nu_{I_{214}(n+j-1),\overline{\partial}}(\Omega_{I_{214}(n+j-1)})$ with
\item [(ii)] ${f_{j}}_{\mid\Omega_{I_{14}(n+j)}}\in\mathcal{E}\nu_{I_{14}(n+j),\overline{\partial}}(\Omega_{I_{14}(n+j)})
             \subset\mathcal{O}\nu_{I_{14}(n+j)}(\Omega_{I_{14}(n+j)})$
\end{enumerate}
such that
\begin{equation}\label{satz6.1}
|f_{j}-f_{j-1}|_{n+j-1}=|f_{j}-f_{j-1}|_{n+j-1,0}<\frac{\varepsilon}{2^{j+1}}
\end{equation}
by \prettyref{thm:dense_proj_lim} and the condition $I_{214}(k)\geq I_{14}(k+1)$ for all $k\in\N$. 
Therefore we obtain for every $k\in\mathbb{N}$
\begin{align}\label{satz6.2}
|f_{k}-f_{0}|_{n}&=\bigl|\sum^{k}_{j=1}{f_{j}-f_{j-1}}\bigr|_{n}
\leq\sum^{k}_{j=1}{|f_{j}-f_{j-1}|_{n}}
\leq\sum^{k}_{j=1}{|f_{j}-f_{j-1}|_{n+j-1}}\nonumber\\
&\underset{\mathclap{\eqref{satz6.1}}}{\leq}\sum^{k}_{j=1}{\frac{\varepsilon}{2^{j+1}}}
=\frac{\varepsilon}{2}\bigl(1-\frac{1}{2^{k}}\bigr)<\frac{\varepsilon}{2}.
\end{align}
Now, let $\varepsilon_{0}>0$ and $l\in\mathbb{N}$. We choose $l_{0}\in\mathbb{N}$, $l_{0}\geq l$, 
such that $\frac{\varepsilon}{2^{l_{0}+1}}<\varepsilon_{0}$. Similarly, we get for all $p\geq k\geq l_{0}$
\begin{align*}
|f_{p}-f_{k}|_{l}&\leq\bigl|f_{p}-f_{k}\bigr|_{l_{0}}
=\bigl|\sum^{p}_{j=k+1}{f_{j}-f_{j-1}}\bigr|_{l_{0}}
\leq\sum^{p}_{j=k+1}{\left|f_{j}-f_{j-1}\right|_{l_{0}}}\\
&\underset{\mathclap{\substack{l_{0}\leq k\leq j-1\\\phantom{l_{0}}<n+j-1}}}{\leq}\quad\;\sum^{p}_{j=k+1}{|f_{j}-f_{j-1}|_{n+j-1}}
\underset{\mathclap{\eqref{satz6.1}}}{\leq}\sum^{p}_{j=k+1}{\frac{\varepsilon}{2^{j+1}}}
=\frac{\varepsilon}{2}\bigl(\frac{1}{2^{k}}-\frac{1}{2^{p}}\bigr)\\
&<\frac{\varepsilon}{2^{k+1}}
\leq\frac{\varepsilon}{2^{l_{0}+1}}<\varepsilon_{0}.
\end{align*}
Hence $(f_{k})_{k\geq n_{0}}$ is a Cauchy sequence in the Banach space $\mathcal{O}\nu_{I_{14}(n+n_{0})}(\Omega_{I_{14}(n+n_{0})})$ 
for every $n_{0}\in\N_{0}$ and thus has a limit $F_{n_{0}}\in\mathcal{O}\nu_{I_{14}(n+n_{0})}(\Omega_{I_{14}(n+n_{0})})$. 
These limits coincide on their common domain because for every $n_{1},n_{2}\in\N_{0}$ with $I_{14}(n+n_{1})<I_{14}(n+n_{2})$ 
and $\varepsilon_{1}>0$ there exists $N\in\N$ such that for all $k\geq N$
\begin{align*}
|F_{n_{1}}-F_{n_{2}}|_{I_{14}(n+n_{1})}&\leq|F_{n_{1}}-f_{k}|_{I_{14}(n+n_{1})}+|f_{k}-F_{n_{2}}|_{I_{14}(n+n_{1})}\\
&\leq|F_{n_{1}}-f_{k}|_{I_{14}(n+n_{1})}+|f_{k}-F_{n_{2}}|_{I_{14}(n+n_{2})}
<\frac{\varepsilon_{1}}{2}+\frac{\varepsilon_{1}}{2}=\varepsilon_{1}.
\end{align*}
We deduce that the glued limit function $f$ given by $f:=F_{n_{0}}$ on $\Omega_{I_{14}(n+n_{0})}$ for all $n_{0}\in\N_{0}$ 
is well-defined and we have 
$f\in\bigcap_{n_{0}\in \N_{0}}\mathcal{O}\nu_{I_{14}(n+n_{0})}(\Omega_{I_{14}(n+n_{0})})
=\mathcal{OV}(\Omega)$ since $I_{14}(n+n_{0})\geq n+n_{0}$. 
By the definition of $f$ there exists $N\in\mathbb{N}$ such that for every $k\geq N$
\[
|f-f_{0}|_{n}\leq|f-f_{k}|_{n}+|f_{k}-f_{0}|_{n}
\underset{n\leq I_{14}(n+0)}{<}\frac{\varepsilon}{2}+|f_{k}-f_{0}|_{n}
\underset{\eqref{satz6.2}}{\leq}\frac{\varepsilon}{2}+\frac{\varepsilon}{2}=\varepsilon
\]
which proves our statement.
\end{proof}
\section{$(\Omega)$ for $\mathcal{OV}$-spaces on strips}
Using \prettyref{cor:weakly_reduced} and a decomposition theorem of Langenbruch, 
we prove that the space $\mathcal{OV}(\Omega)$ where the 
$\Omega_{n}$ are strips along the real axis satifies the property $(\Omega)$ of Vogt 
for suitable weights $\mathcal{V}$. 
Let us recall that a Fr\'echet space $F$ with an increasing fundamental system of 
seminorms $(\norm{3}{\cdot}_{k})_{k\in\N}$ satisfies $(\Omega)$ if
\begin{equation}\label{om1}
\forall\; p\in\N\; \exists\; q\in\N\;\forall\; k\in\N\;\exists\; n\in\N,\,C>0\;\forall\; r>0:\;
 U_{q}\subset Cr^n U_k + \frac{1}{r} U_p
\end{equation}
where $U_{k}:=\{x\in F \; | \; \norm{3}{x}_{k}\leq 1\}$ (see \cite[Chap.\ 29, Definition, p.\ 367]{meisevogt1997}). 
The weights we want to consider are generated by a function $\mu$ with the following properties.

\begin{defn}[{(strong) weight generator}]\label{def:weight_generator}
A continuous function $\mu\colon \C\to[0,\infty)$ is called a weight generator if $\mu(z)=\mu(|\re(z)|)$ for all $z\in\C$, 
the restriction $\mu_{\mid [0,\infty)}$ is strictly increasing, 
\[
 \lim_{\substack{x\to\infty\\x\in\R}}\frac{\ln(1+|x|)}{\mu(x)}=0
\]
and 
\[
 \exists\;\Gamma>1,\,C>0\;\forall\;x\in[0,\infty):\; \mu(x+1)\leq \Gamma \mu(x)+C.
\]
If $\mu$ is a weight generator which fulfils the stronger condition 
\[
 \exists\;\Gamma>1\;\forall\;n\in\N\;\exists\;C>0\;\forall\;x\in[0,\infty):\; \mu(x+n)\leq \Gamma \mu(x)+C,
\]
then $\mu$ is called a strong weight generator.
\end{defn}

Weight generators are introduced in \cite[Definition 2.1, p.\ 225]{L1} and strong weight generators in 
\cite[Definition 2.2.2, p.\ 43]{Toenjes} where they are simply called weight functions resp.\ strong weight functions.
For a weight generator $\mu$ we define the space 
\[
H_{\tau}(S_{t}):=\{f\in\mathcal{O}(S_{t})\;|\;\|f\|_{\tau,t}:=\sup_{z\in S_{t}}|f(z)|e^{\tau\mu(z)}<\infty\}
\]
for $t>0$ and $\tau\in\R$ with the strip $S_{t}:=\{z\in\C\;|\;|\im(z)|<t\}$ .

\begin{thm}[{\cite[Theorem 2.2, p.\ 225]{L1}}]\footnote{A superfluous constant depending on 
$\operatorname{sign}(\tau_{0})$ is omitted.}\label{thm:decomposition}
Let $\mu$ be a weight generator. There are $\widetilde{t}$, $K_{1}$, $K_{2}>0$ such that 
for any $\tau_{0}<\tau<\tau_{2}$ there is $C_{0}=C_{0}(\operatorname{sign}(\tau))$ 
such that for any $0<2t_{0}<t<t_{2}<\widetilde{t}$ with
\[
t_{0}\leq\min\Bigl[K_{1},K_{2}\sqrt{\frac{\tau-C_{0}\tau_{0}}{\tau_{2}-C_{0}\tau_{0}}}\Bigr]
\]
there is $C_{1}\geq 1$ such that for any $r\geq 0$ and any $f\in H_{\tau}(S_{t})$ with $\|f\|_{\tau,t}\leq 1$ the following holds: 
there are $f_{2}\in \mathcal{O}(S_{t_{2}})$ and $f_{0}\in \mathcal{O}(S_{t_{0}})$ such that $f=f_{0}+f_{2}$ on $S_{t_{0}}$ and
\[
\|f_{0}\|_{C_{0}\tau_{0},t_{0}}\leq C_{1}e^{-Gr}\quad\text{and}\quad\|f_{2}\|_{\tau_{2},t_{2}}\leq e^{r}
\]
where
\[
G:= K_{1} \min\Bigl[1,\frac{t-t_{0}}{2\widetilde{t}},\frac{\tau-C_{0}\tau_{0}}{\tau_{2}-C_{0}\tau_{0}}\Bigr].
\]
\end{thm}

To apply this theorem, we have to know the constants involved. In the following the notation of \cite{L1} is used and it is referred 
to the corresponding positions resp.\ conditions for these constants. We have
\[
\widetilde{t}:=\frac{1}{4 \ln(\Gamma)}
\]
by \cite[Lemma 2.4, (2.15), p.\ 228]{L1} with $\Gamma$ from \prettyref{def:weight_generator} such that $\Gamma\geq e^{1/4}$. 
The choice $\Gamma\geq e^{1/4}$ comes from wanting $\widetilde{t}\leq 1$ in \cite[Lemma 2.4, p.\ 228]{L1}.
By \cite[Corollary 2.6, p.\ 230-231]{L1} we have
\[
C_{0}:=
\begin{cases}4 \Gamma B_{3}=\frac{64\cosh(1)}{\cos(1/2)}\Gamma^{2}>1 &,\tau<0,\\
\frac{1}{4 \Gamma B_{3}}=\frac{\cos(1/2)}{64\cosh(1)\Gamma^{2}}<1 &,\tau\geq 0,
\end{cases}
\]
where $B_{3}:=\frac{16\cosh(1)}{\cos(1/2)}\Gamma$ by \cite[Lemma 2.4, p.\ 228-229]{L1}.\footnote{An error in part b) of this lemma, 
p.\ 229, is corrected here such that the term $\cos(1/2)=\min_{|y|\leq \widetilde{t}=1/(2C_{1})}\cos(C_{1}y)$ appears.} 
To get the constants $K_{1}$ and $K_{2}$, we have to analyze the conditions for $t_{0}$ in the proof of \cite[Theorem 2.2, p.\ 225]{L1}. 
By the assumptions on $\tau_{0}$, $\tau$ and $\tau_{2}$ and the choice of $C_{0}$ we obtain
\begin{equation}\label{thm12.1}
\tau_{2}-C_{0}\tau_{0}>\tau_{2}-C_{0}\tau\geq \tau_{2}-\tau>0
\end{equation}
and
\begin{equation}\label{thm12.2}
\tau-C_{0}\tau_{0}>\tau-C_{0}\tau=\tau(1-C_{0})>0.
\end{equation}
By choosing $D>0$ in the proof of \cite[Theorem 2.2, (2.22), p.\ 232-233]{L1} as
$D:=\frac{\tau-C_{0}\tau_{0}}{(\tau_{2}-C_{0}\tau_{0})2\Gamma_{0}}$, 
the estimate
\[
D=\frac{\tau-C_{0}\tau_{0}}{(\tau_{2}-C_{0}\tau_{0})2\Gamma_{0}}
 =\min\Bigl(\frac{1}{2\widetilde{\Gamma}},\frac{1}{2\widehat{\Gamma}}\Bigr)\frac{\tau-C_{0}\tau_{0}}{\tau_{2}-C_{0}\tau_{0}}
 \underset{\eqref{thm12.1},\,\eqref{thm12.2}}{\leq}\min\Bigl(\frac{1}{2\widetilde{\Gamma}},\frac{1}{2\widehat{\Gamma}}\Bigr)
 \frac{\tau-C_{0}\tau_{0}}{\tau_{2}-C_{0}\tau}
\]
holds where $\Gamma_{0}:=\max(\widetilde{\Gamma},\widehat{\Gamma})$ with $\widetilde{\Gamma}$, $\widehat{\Gamma}>1$ from the proof. 
With $\theta\geq\frac{t-t_{0}}{2\widetilde{t}}$ (p.\ 232) we get on p.\ 233, below (2.24), due to the condition 
$t_{0}\leq T_{0}:=\min(\frac{t}{2},\frac{1}{4a^{2}B_{1}\widetilde{t}})$,
\begin{align*}
 \min\Bigl(\frac{\theta}{2}, D, 1\Bigr)
&\geq \min\Bigl(\frac{1}{2},\frac{1}{2\Gamma_{0}}\Bigr)
 \min\Bigl(\theta,\frac{\tau-C_{0}\tau_{0}}{\tau_{2}-C_{0}\tau_{0}},1\Bigr)
 \geq\frac{1}{2\Gamma_{0}}\min\Bigl(\frac{t-t_{0}}{2\widetilde{t}},\frac{\tau-C_{0}\tau_{0}}{\tau_{2}-C_{0}\tau_{0}},1\Bigr)\\
&\geq \min\Bigl(\frac{1}{2\Gamma_{0}},\frac{1}{4a^{2}B_{1}\widetilde{t}}\Bigr)\min\Bigl(\frac{t-t_{0}}{2\widetilde{t}},
 \frac{\tau-C_{0}\tau_{0}}{\tau_{2}-C_{0}\tau_{0}},1\Bigr)\\
&=\underbrace{\min\Bigl(\frac{1}{2\Gamma_{0}},\frac{1}{2\cosh(1)\ln(\Gamma)}\Bigr)}_{=:K_{1}} 
  \min\Bigl(\frac{t-t_{0}}{2\widetilde{t}},\frac{\tau-C_{0}\tau_{0}}{\tau_{2}-C_{0}\tau_{0}},1\Bigr)=:G
\end{align*}
where $a:=\ln(\Gamma)$ (in the middle of p.\ 231) and $B_{1}:=2\cosh(1)$ by the proof of \cite[Lemma 2.3, p.\ 226-227]{L1}.
The assumptions $2t_{0}<t$ and $t_{0}\leq K_{1}$ in \prettyref{thm:decomposition} guarantee 
that the condition $t_{0}\leq T_{0}$ is satisfied. 
Looking at the condition $t_{0}\leq T_{1}:=\sqrt{\frac{D}{a^{2}B_{1}}}$ (p.\ 232), we derive
\[
 T_{1}=\frac{1}{\sqrt{2\Gamma_{0}a^{2}B_{1}}}\sqrt{\frac{\tau-C_{0}\tau_{0}}{\tau_{2}-C_{0}\tau_{0}}}
=\underbrace{\frac{1}{2\sqrt{\cosh(1)\Gamma_{0}}\ln(\Gamma)}}_{=:K_{2}}\sqrt{\frac{\tau-C_{0}\tau_{0}}{\tau_{2}-C_{0}\tau_{0}}}.
\]

For the subsequent theorem we merge and modify the proofs of \cite[Satz 2.2.3, p.\ 44]{Toenjes}
\footnote{The proof of \cite[Satz 2.2.3, p.\ 44]{Toenjes} relies on \cite[Satz 2.2.1, p.\ 43]{Toenjes} 
which is an announced version (without a proof) of our density result \prettyref{cor:weakly_reduced}.} 
($a_{n}=n$, $n\in\N$, and $\mu$ a strong weight generator) and 
\cite[5.20 Theorem, p.\ 84]{ich} ($a_{n}=-1/n$, $n\in\N$, and $\mu=|\re(\cdot)|$). 

\begin{thm}\label{thm:omega_for_strips}
Let $\mu$ be a strong weight generator, $(a_{n})_{n\in\N}$ strictly increasing, $a_{n}<0$ for all $n\in\N$ 
or $a_{n}\geq 0$ for all $n\in\N$, $\lim_{n\to\infty}a_{n}=0$ or $\lim_{n\to\infty}a_{n}=\infty$, 
$\mathcal{V}:=(\exp(a_{n}\mu))_{n\in\N}$ and $\Omega_{n}:=S_{n}$ for all $n\in\N$.
If \prettyref{cond:dense} with $I_{214}(n)\geq I_{14}(n+1)$ for all $n\in\N$ and $(\omega.1)$ are fulfilled, 
then $\mathcal{OV}(\C)$ satisfies $(\Omega)$.
\end{thm}
\begin{proof}
Let $p\in\N$. As $(a_{n})_{n\in\N}$ is strictly increasing and $\lim_{n\to\infty}a_{n}=0$ or $\lim_{n\to\infty}a_{n}=\infty$, 
we may choose $q\in\N$ such that $a_{2J_{1}I_{14}(p)}/C_{0}<a_{q}$ and $4J_{1}I_{14}(p)<q$. 
To use the theorem above, we need a linear transformation between strips to get the decomposition on the desired strip, 
desired in the spirit of \prettyref{cor:weakly_reduced}. We choose $\Gamma\geq e^{1/4}$ and $T\in\R$ such that
\begin{equation}\label{thm13.0}
 0<T<\frac{1}{4\max(q+1,2J_{1}I_{14}(k))\ln(\Gamma)}
\end{equation}
which also fulfils 
\begin{align}\label{thm13.1}
T&\leq\frac{1}{2J_{1}I_{14}(p)}\min\Biggl(\frac{1}{2\Gamma_{0}},\frac{1}{2\cosh(1)\ln(\Gamma)},\notag\\
&\phantom{\frac{1}{2J_{1}I_{14}(p)}\min\Biggl(}\;\;\frac{1}{2\sqrt{\cosh(1)\Gamma_{0}}\ln(\Gamma)}
\sqrt{\frac{a_{q}-a_{2J_{1}I_{14}(p)}}{\max(a_{q+1},a_{2J_{1}I_{14}(k)})-a_{2J_{1}I_{14}(p)}}}\Biggr).
\end{align}
Let
\begin{align*}
\tau_{0}&:=\frac{a_{2J_{1}I_{14}(p)}}{C_{0}}, && \tau:=a_{q}, &&\tau_{2}:=\max(a_{q+1},a_{2J_{1}I_{14}(k)}),\\
t_{0}&:= 2J_{1}I_{14}(p)T,  &&t:=qT,   &&t_{2}:=\max(q+1,2J_{1}I_{14}(k))T.
\end{align*}
By the choice of $q$ we have
\[
\tau_{0}=\frac{a_{2J_{1}I_{14}(p)}}{C_{0}}<a_{q}=\tau<\max(a_{q+1},a_{2J_{1}I_{14}(k)})=\tau_{2}.
\]
By the choice of $q$ and \eqref{thm13.0} we get
\[
0<2t_{0}=4J_{1}I_{14}(p)T<qT=t<\max(q+1,2J_{1}I_{14}(k))T=t_{2}<\frac{1}{4\ln(\Gamma)}=\widetilde{t}.
\]
Further, we deduce from \eqref{thm13.1} that 
\[
 t_{0}=2J_{1}I_{14}(p)T\leq\min\Bigl[K_{1},K_{2}\sqrt{\frac{\tau-C_{0}\tau_{0}}{\tau_{2}-C_{0}\tau_{0}}}\Bigr].
\]
Let $r\geq 0$ and $f\in\mathcal{OV}(\C)$ such that $|f|_{q}=\|f\|_{a_{q},q}\leq 1$. We set $\widetilde{f}\colon S_{qT}\to \C$, 
$\widetilde{f}(z):=f(z/T)$, and define
\[
 H_{\tau}^{\sim}(S_{t}):=\{g\in\mathcal{O}(S_{t})\;|\;\|g\|_{\tau,t}^{\sim}:=\sup_{z\in S_{t}}|g(z)|e^{\tau\widetilde{\mu}(z)}<\infty\}
\]
where $\widetilde{\mu}:=\mu(\cdot/T)$. We note that for $\widetilde{n}:=\lceil 1/T\rceil$, 
where $\lceil\cdot\rceil$ is the ceiling function, there is $C>0$ such that for all $x\geq 0$
\[
 \widetilde{\mu}(x+1)=\mu\Bigl(\frac{x+1}{T}\Bigr)
 \leq\mu\Bigl(\frac{x}{T}+\bigl\lceil\frac{1}{T}\bigr\rceil\Bigr)
 =\mu\Bigl(\frac{x}{T}+\widetilde{n}\Bigr)
 \leq\Gamma \mu\Bigl(\frac{x}{T}\Bigr)+C
 =\Gamma\widetilde{\mu}(x)+C
\]
because $\mu$ is a strong weight generator. We conclude that $\widetilde{\mu}$ is also a weight generator 
with the same $\Gamma$ as $\mu$ which is independent of $T$. Moreover, from
\[
\|\widetilde{f}\|_{\tau,t}^{\sim}=\sup_{z\in S_{qT}}|\widetilde{f}(z)|e^{a_{q}\widetilde{\mu}(z)}
=\sup_{z\in S_{q}}|f(z)|e^{a_{q}\mu(z)}=|f|_{q}\leq 1
\]
follows by \prettyref{thm:decomposition} that there are $\widetilde{f}_{j}\in\mathcal{O}(S_{t_{j}})$, $j\in\{0,2\}$, such that
\begin{equation}\label{thm13.2}
\widetilde{f}(z)=\widetilde{f}_{0}(z)+\widetilde{f}_{2}(z),\quad z\in S_{t_{0}},
\end{equation}
and
\begin{align}\label{thm13.3}
 C_{1}e^{-Gr}
&\geq\|\widetilde{f}_{0}\|_{C_{0}\tau_{0},t_{0}}^{\sim}
 =\sup_{z\in S_{t_{0}}}|\widetilde{f}_{0}(z)|e^{C_{0}\tau_{0}\widetilde{\mu}(z)}
 =\sup_{z\in S_{t_{0}/T}}|\underbrace{\widetilde{f}_{0}(Tz)}_{=:f_{0}(z)}|e^{C_{0}\tau_{0}\widetilde{\mu}(Tz)}\notag\\
&=\sup_{z\in S_{2J_{1}I_{14}(p)}}|f_{0}(z)|e^{a_{2J_{1}I_{14}(p)}\mu(z)}=|f_{0}|_{2J_{1}I_{14}(p)},
\end{align}
where $f_{0}\in\mathcal{O}(S_{2J_{1}I_{14}(p)})$, as well as
\begin{align}\label{thm13.4}
 e^{r}
&\geq\|\widetilde{f}_{2}\|_{\tau_{2},t_{2}}^{\sim}
 =\sup_{z\in S_{t_{2}}}|\widetilde{f}_{2}(z)|e^{\tau_{2}\widetilde{\mu}(z)}
 =\sup_{z\in S_{t_{2}/T}}|\underbrace{\widetilde{f}_{2}(Tz)}_{=:f_{2}(z)}|e^{\tau_{2}\widetilde{\mu}(Tz)}\notag\\
&\geq\sup_{z\in S_{2J_{1}I_{14}(k)}}|f_{2}(z)|e^{a_{2J_{1}I_{14}(k)}\mu(z)}=|f_{2}|_{2J_{1}I_{14}(k)}
\end{align}
where $f_{2}\in\mathcal{O}(S_{t_{2}/T})\subset\mathcal{O}(S_{2J_{1}I_{14}(k)})$ and the inclusion is justified by the identity theorem. 
Furthermore, for $z\in S_{t_{0}/T}=S_{2J_{1}I_{14}(p)}$ the equation
\[
f(z)=\widetilde{f}(Tz)\underset{\eqref{thm13.2}}{=}\widetilde{f}_{0}(Tz)+\widetilde{f}_{2}(Tz)=f_{0}(z)+f_{2}(z)
\]
holds, thus $f=f_{0}+f_{2}$ on $S_{2J_{1}I_{14}(p)}$. By virtue of \prettyref{cor:weakly_reduced} the following is valid:
\begin{equation}\label{thm13.5}
\forall\;\varepsilon>0\;\exists\;\widehat{f}_{0},\,\widehat{f}_{2}\in\mathcal{OV}(\C):\;
(i)\;\;|\widehat{f}_{0}-f_{0}|_{p}<\varepsilon \quad\text{and}\quad
(ii)\;\;|\widehat{f}_{2}-f_{2}|_{k}<\varepsilon.
\end{equation}
Now, we have to consider two cases. Let $\varepsilon:=C_{1}e^{-Gr}$. For $k\leq p$ we get via \eqref{thm13.5} $(i)$
\[
f=\widehat{f}_{0}+(f_{2}+f_{0}-\widehat{f}_{0})\quad\text{on}\;S_{2J_{1}I_{14}(p)},
\]
so
\begin{equation}\label{thm13.6}
f_{2}+f_{0}-\widehat{f}_{0}=f-\widehat{f}_{0}=:\overline{f}_{2}\quad\text{on}\;S_{2J_{1}I_{14}(p)}
\end{equation}
where the function $\overline{f}_{2}\in\mathcal{OV}(\C)$ and thus is a holomorphic extension of 
the left-hand side on $\C$. Hence we clearly have $f=\widehat{f}_{0}+\overline{f}_{2}$ and
\begin{equation}\label{thm13.7}
 |\widehat{f}_{0}|_{p}
\leq|\widehat{f}_{0}-f_{0}|_{p}+|f_{0}|_{p}
 \underset{\eqref{thm13.5} (i)}{\leq}\varepsilon+|f_{0}|_{p}
\leq\varepsilon+|f_{0}|_{2J_{1}I_{14}(p)}
\underset{\eqref{thm13.3}}{\leq}2C_{1}e^{-Gr}=:C_{2}e^{-Gr}
\end{equation}
as well as
\begin{align}\label{thm13.8}
 |\overline{f}_{2}|_{k}
&\leq |\overline{f}_{2}-f_{2}|_{k}+|f_{2}|_{k}
 \underset{\eqref{thm13.6},\,k\leq p}{\leq}|f_{0}-\widehat{f}_{0}|_{p}+|f_{2}|_{2J_{1}I_{14}(k)}
 \underset{\eqref{thm13.5} (i)}{\leq} \varepsilon+|f_{2}|_{2J_{1}I_{14}(k)}\notag\\
&\underset{\mathclap{\eqref{thm13.4}}}{\leq} C_{1}e^{-Gr}+e^{r}
 \leq(C_{1}+1)e^{r}=:C_{3}e^{r}.
\end{align}
Analogously, for $k>p$ we obtain via \eqref{thm13.5} $(ii)$
\[
f=\widehat{f}_{2}+(f_{0}+f_{2}-\widehat{f}_{2})\quad\text{on}\;S_{2J_{1}I_{14}(p)},
\]
so
\begin{equation}\label{thm13.9}
f_{0}+f_{2}-\widehat{f}_{2}=f-\widehat{f}_{2}=:\overline{f}_{0}\quad\text{on}\;S_{2J_{1}I_{14}(p)}
\end{equation}
where the function $\overline{f}_{0}\in\mathcal{OV}(\C)$ and thus is a holomorphic extension of the left-hand side on $\C$. 
Hence we clearly have $f=\overline{f}_{0}+\widehat{f}_{2}$ and
\begin{align}\label{thm13.10}
 |\overline{f}_{0}|_{p}
&=|f-\widehat{f}_{2}|_{p}
 \underset{\eqref{thm13.9}}{=}|f_{0}+f_{2}-\widehat{f}_{2}|_{p}
 \leq |f_{2}-\widehat{f}_{2}|_{p}+|f_{0}|_{p}
 \underset{k>p}{\leq}|f_{2}-\widehat{f}_{2}|_{k}+|f_{0}|_{2J_{1}I_{14}(p)}\notag\\
&\underset{\mathclap{\eqref{thm13.5}(ii)}}{\leq}\varepsilon+|f_{0}|_{2J_{1}I_{14}(p)}
 \underset{\eqref{thm13.3}}{\leq} 2C_{1}e^{-Gr}=C_{2}e^{-Gr}
\end{align}
as well as
\begin{equation}\label{thm13.11}
 |\widehat{f}_{2}|_{k}
\leq|\widehat{f}_{2}-f_{2}|_{k}+|f_{2}|_{k}
 \underset{\eqref{thm13.5} (ii)}{\leq}\varepsilon+|f_{2}|_{2J_{1}I_{14}(k)}
 \underset{\eqref{thm13.4}}{\leq} C_{1}e^{-Gr}+e^{r}\leq C_{3}e^{r}.
\end{equation}
Next, we set $n:=\lceil 1/G \rceil$ and $C:=C_{3}e^{\ln(C_{2})/G}$. Let $\widetilde{r}>0$. 
For $\widetilde{r}\geq 1$ there is $r\geq 0$ such that
\[
\widetilde{r}=e^{Gr-\ln(C_{2})}=\frac{e^{Gr}}{C_{2}}
\]
and we have by \eqref{thm13.7} and \eqref{thm13.8} for $k\leq p$
\[
|\widehat{f}_{0}|_{p}\leq C_{2}e^{-Gr}=\frac{1}{\widetilde{r}},
\quad
|\overline{f}_{2}|_{k}
\leq C_{3}e^{r}=C_{3}e^{\frac{1}{G}\ln(C_{2})}e^{\frac{1}{G}(Gr-\ln(C_{2}))}
=C\,\widetilde{r}^{\,\frac{1}{G}}\underset{\widetilde{r}\geq 1}{\leq}C\,\widetilde{r}^{\;n},
\]
as well as by \eqref{thm13.10} and \eqref{thm13.11} for $k>p$
\[
|\overline{f}_{0}|_{p}\leq\frac{1}{\widetilde{r}},\quad |\widehat{f}_{2}|_{k}\leq C\,\widetilde{r}^{\;n}.
\]
For $0<\widetilde{r}<1$ we have, since $q\geq p$,
\[
|f|_{p}\leq|f|_{q}\leq 1<\frac{1}{\widetilde{r}}.
\]
Thus our statement is proved.
\end{proof}

Let us remark that the choice of the sequence $(a_{n})_{n\in\N}$ in the preceding theorem does not really matter. 

\begin{rem}
Let $\mu\colon\C\to [0,\infty)$ be continuous, $(a_{n})_{n\in\N}$ strictly increasing, $a_{n}<0$ for all $n\in\N$ 
or $a_{n}\geq 0$ for all $n\in\N$, $\lim_{n\to\infty}a_{n}=0$ or $\lim_{n\to\infty}a_{n}=\infty$, 
$\mathcal{V}:=(\exp(a_{n}\mu))_{n\in\N}$ and $\Omega_{n}:=S_{n}$ for all $n\in\N$. 
Set $\mathcal{V}_{-}:=(\exp((-1/n)\mu))_{n\in\N}$ and $\mathcal{V}_{+}:=(\exp(n\mu))_{n\in\N}$. 
Then 
\[
\mathcal{OV}(\C)\cong \mathcal{OV}_{-}(\C),\quad \text{if}\;a_{n}<0,\quad\text{and}\quad
\mathcal{OV}(\C)\cong \mathcal{OV}_{+}(\C),\quad \text{if}\;a_{n}\geq 0,
\]
which is easily seen. Thus one may choose the most suitable sequence $(a_{n})_{n\in\N}$ for one's purpose
without changing the space.
\end{rem}

\begin{cor}\label{cor:omega_for_strips}
Let $(a_{n})_{n\in\N}$ be strictly increasing, $a_{n}<0$ for all $n\in\N$ 
or $a_{n}\geq 0$ for all $n\in\N$, $\lim_{n\to\infty}a_{n}=0$ or $\lim_{n\to\infty}a_{n}=\infty$, 
$\mathcal{V}:=(\exp(a_{n}\mu))_{n\in\N}$ and $\Omega_{n}:=S_{n}$ for all $n\in\N$
where 
\[
 \mu\colon\C \to [0,\infty),\;\mu(z):=|\re(z)|^{\gamma},
\]
for some $0<\gamma\leq 1$. Then $\mathcal{OV}(\C)$ satisfies $(\Omega)$.
\end{cor}
\begin{proof}
We only need to check that the conditions of \prettyref{thm:omega_for_strips} are fulfilled. 
Obviously, $\mu(z)=\mu(|\re(z)|)$ for all $z\in\C$, $\mu$ is strictly increasing on $[0,\infty)$ 
and $\lim_{x\to\infty,\, x\in\R}\frac{\ln(1+|x|)}{\mu(x)}=0$. The observation 
 \[
  \mu(x+n)-\mu(x)=|x+n|^{\gamma}-|x|^{\gamma}\leq |x+n-n|^{\gamma}=n^{\gamma},\quad n\in\N,\;x\in[0,\infty),
 \]  
implies that $\mu$ is a strong weight generator with any $\Gamma>1$ and $C:=n^{\gamma}$ 
by \prettyref{def:weight_generator}. 
In addition, condition $(\omega.1)$ is fulfilled by \prettyref{ex:weights_little_omega} a). 
Let us turn to \prettyref{cond:dense}. If $a_{n}<0$ for all $n\in\N$, then \prettyref{cond:dense} 
is fulfilled by \cite[4.10 Example a), p.\ 22]{kruse2018_5} where we used $\widetilde{\mu}(z):=|z|^{\gamma}$ 
instead of $\mu$ which does not make a difference since 
 \[
  |\re(z)|^{\gamma}\leq |z|^{\gamma}\leq |\re(z)|^{\gamma}+n^{\gamma},\quad z\in\Omega_{n}=S_{n}.
 \] 
If $a_{n}\geq 0$ for all $n\in\N$, 
we only have to modify \cite[4.10 Example a), p.\ 22]{kruse2018_5} a bit. 
We choose $I_{j}(n):=2n$ for $j\in\{1,2,4\}$ and define the open set $X_{I_{2}(n)}:=\overline{S}_{4n}^{C}$. 
Then we have
 \[
  I_{214}(n)=8n\geq 4n+4=I_{14}(n+1),\quad n\in\N.
 \]
Furthermore, we have $\d_{n,k}=|n-k|$ for all $n,k\in\N$.

 \prettyref{cond:dense} a)(i) and c): Verbatim as in \cite[4.10 Example a), p.\ 22]{kruse2018_5}.
 
 \prettyref{cond:dense} a)(ii): We have $\d_{X,I_{2}}=2n$. 
 We choose $g_{n}\colon\C\to\C$, $g_{n}(z):=\exp(-z^2)$, 
 as well as $r_{n}:=1/(4n)$ and $R_{n}:=1/(6n)$ for $n\in\N$. 
 Let $z=z_{1}+iz_{2}\in\Omega_{I_{2}(n)}=S_{2n}$ and $x\in X_{I_{2}(n)}+\mathbb{B}_{R_{n}}(0)$. 
 For $\zeta=\zeta_{1}+i\zeta_{2}\in\C$ with $|\zeta-(z-x)|=r_{n}$ we have 
 \begin{align*}
 |g_{n}(\zeta)|e^{a_{2n}\mu(z)}&=e^{-\re(\zeta^{2})}e^{a_{2n}|\re(z)|^{\gamma}}
  \leq e^{-\zeta_{1}^{2}+\zeta_{2}^{2}}e^{a_{2n}(1+|z_{1}|)}\\
 &\leq e^{(r_{n}+|z_{2}|+|x_{2}|)^{2}+a_{2n}(1+r_{n}+|x_{1}|)}e^{-|\zeta_{1}|^{2}+a_{2n}|\zeta_{1}|}\\
 &\leq e^{(r_{n}+2n+|x_{2}|)^{2}+a_{2n}(1+r_{n}+|x_{1}|)}\sup_{t\in\R}e^{-t^{2}+a_{2n}t}\\
 &= e^{(r_{n}+2n+|x_{2}|)^{2}+a_{2n}(1+r_{n}+|x_{1}|)+a_{2n}^{2}/4}=:A_{2}(x,n)
 \end{align*}
 and observe that $A_{2}(\cdot,n)$ is continuous and thus locally bounded on $X_{I_{2}(n)}$.
 
 \prettyref{cond:dense} a)(iii): Let $K\subset\C$ be compact and $x=x_{1}+ix_{2}\in\Omega_{n}$. Then there 
 is $b>0$ such that $|y|\leq b$ for all $y=y_{1}+iy_{2}\in K$ and from polar coordinates and Fubini's theorem 
 follows that
 \begin{flalign*}
 &\hspace{0.37cm}\int_{K}\frac{|g_{n}(x-y)|}{|x-y|}\d y\\
 &\leq\underbrace{\sup_{w\in K}e^{a_{2n}|\re(w)|}}_{=:C_{1}}
  \int_{K}\frac{e^{-\re((x-y)^{2})}}{|x-y|}e^{-a_{2n}|y_{1}|}\d y\\
 &\leq C_{1}\bigl( \int_{\mathbb{B}_{1}(x)}\frac{e^{-\re((x-y)^{2})}}{|x-y|}e^{-a_{2n}|\re(y)|}\d y
  +\int_{K\setminus \mathbb{B}_{1}(x)}\frac{e^{-\re((x-y)^{2})}}{|x-y|}e^{-a_{2n}|\re(y)|}\d y\bigr)\\
 &\leq C_{1}\bigl(\int_{0}^{2\pi}\int_{0}^{1}\frac{e^{-r^{2}\cos(2\varphi)}}{r}e^{-a_{2n}|x_{1}+r\cos(\varphi)|}r\d r\d\varphi
  + \int_{K\setminus \mathbb{B}_{1}(x)}e^{-\re((x-y)^{2})}e^{-a_{2n}|\re(y)|}\d y\bigr)\\
 &\leq C_{1}\bigl(2\pi e^{1+a_{2n}}e^{-a_{2n}|x_{1}|}
      +\int_{-b}^{b}e^{(x_{2}-y_{2})^{2}}\d y_{2}
       \int_{\R}e^{-(x_{1}-y_{1})^{2}+a_{2n}|x_{1}-y_{1}|}\d y_{1}e^{-a_{2n}|x_{1}|}\bigr)\\
 &\leq C_{1}\bigl(2\pi e^{1+a_{2n}}+2be^{(|x_{2}|+b)^{2}}\int_{\R}e^{-y_{1}^{2}+a_{2n}|y_{1}|}\d y_{1}\bigr)e^{-a_{2n}|x_{1}|}\\
 &= C_{1}\bigl(2\pi e^{1+a_{2n}}+2be^{(|x_{2}|+b)^{2}}e^{a_{2n}^{2}/4}
       \int_{\R}e^{-(|y_{1}|-a_{2n}/2)^{2}}\d y_{1}\bigr)e^{-a_{2n}|x_{1}|}\\
 &= C_{1}\bigl(2\pi e^{1+a_{2n}}+4be^{(|x_{2}|+b)^{2}}e^{a_{2n}^{2}/4}
       \int_{-a_{2n}/2}^{\infty}e^{-y_{1}^{2}}\d y_{1}\bigr)e^{-a_{2n}|x_{1}|}\\      
 &\leq  C_{1}\bigl(2\pi e^{1+a_{2n}}+4\sqrt{\pi}be^{(n+b)^{2}+a_{2n}^{2}/4}\bigr)e^{-a_{2n}|x_{1}|}.
 \end{flalign*}
 We conclude that \prettyref{cond:dense} a)(iii) holds since 
 \[
 e^{-a_{2n}|x_{1}|}e^{a_{n}|\re(x)|^{\gamma}}\leq e^{(a_{n}-a_{2n})|x_{1}|+a_{n}}\leq e^{a_{n}}.
 \]
 
 \prettyref{cond:dense} b): Let $p,k\in\N$ with $p\leq k$. For all $x=x_{1}+ix_{2}\in\Omega_{p}$ 
 and $y=y_{1}+iy_{2}\in\Omega_{I_{4}(n)}$ we note that
 \[
       a_{p}|\re(x)|^{\gamma}-a_{k}|\re(y)|^{\gamma}
  \leq a_{k}|x_{1}-y_{1}|^{\gamma}\leq a_{k}(1+|x_{1}-y_{1}|)
 \]
 because $(a_{n})_{n\in\N}$ is non-negative and increasing and $0<\gamma\leq 1$. Like before we deduce that
 \begin{flalign*}
 &\hspace{0.37cm}\int_{\Omega_{I_{4}(n)}}\frac{|g_{n}(x-y)|\nu_{p}(x)}{|x-y|\nu_{k}(y)}\d y\\
 &=\int_{\Omega_{2n}}\frac{e^{-\re((x-y)^{2})}}{|x-y|}e^{a_{p}|\re(x)|^{\gamma}-a_{k}|\re(y)|^{\gamma}}\d y
  \leq \int_{\Omega_{2n}}\frac{e^{-\re((x-y)^{2})}}{|x-y|}e^{a_{k}|\re(x)-\re(y)|^{\gamma}}\d y\\
 &\leq \int_{0}^{2\pi}\int_{0}^{1}\frac{e^{-r^{2}\cos(2\varphi)}}{r}e^{a_{k}r^{\gamma}}r\d r\d\varphi
  + \int_{\Omega_{2n}\setminus \mathbb{B}_{1}(x)}e^{-\re((x-y)^{2})}e^{a_{k}|\re(x)-\re(y)|^{\gamma}}\d y\\
 &\leq 2\pi e^{1+a_{k}}+e^{a_{k}}\int_{-2n}^{2n}e^{(x_{2}-y_{2})^{2}}\d y_{2}
        \int_{\R}e^{-(x_{1}-y_{1})^{2}+a_{k}|x_{1}-y_{1}|}\d y_{1}\\
 &\leq 2\pi e^{1+a_{k}}+8\sqrt{\pi}ne^{a_{k}+(|x_{2}|+2n)^{2}+a_{k}^{2}/4}\\
 &\leq 2\pi e^{1+a_{I_{14}(n)}}+8\sqrt{\pi}ne^{a_{I_{14}(n)}+(I_{14}(n)+2n)^{2}+a_{I_{14}(n)}^{2}/4}
\end{flalign*}
for $(k,p)=(I_{4}(n),n)$ and $(k,p)=(I_{14}(n),I_{14}(n))$ as $(a_{n})_{n\in\N}$ is non-negative and increasing.
\end{proof}

\section{Surjectivity of the Cauchy-Riemann operator}
In our last section we prove our main result on the surjectivity of the Cauchy-Riemann operator 
on $\mathcal{EV}(\C,E)$ where $\Omega_{n}:=\{z\in\C\;|\;|\im(z)|<n\}$ for all $n\in\N$. 
We recall the corresponding result for $E=\C$ which we will need. It is a consequence 
of the approximation \prettyref{thm:dense_proj_lim} in combination with H\"ormander's solution of the 
$\overline{\partial}$-problem in weighted $L^{2}$-spaces \cite[Theorem 4.4.2, p.\ 94]{H3} and the Mittag-Leffler procedure. 

\begin{thm}[{\cite[4.8 Theorem, p.\ 20]{kruse2018_5}}]\label{thm:scalar_CR_surjective}
Let \prettyref{cond:weights} with $\psi_{n}(z):=(1+|z|^{2})^{-2}$, $z\in\Omega$, and 
\prettyref{cond:dense} with $I_{214}(n)\geq I_{14}(n+1)$ be fulfilled and $-\ln\nu_{n}$ be subharmonic 
on $\Omega$ for every $n\in\N$. Then
\[
\overline{\partial}\colon \mathcal{EV}(\Omega)\to\mathcal{EV}(\Omega)
\]
is surjective.
\end{thm}

An application of this theorem yields the following corollary.

\begin{cor}[{\cite[4.10 Example a), p.\ 22]{kruse2018_5}}]\label{cor:scalar_surj_CR_example}
Let $(a_{n})_{n\in\N}$ be strictly increasing, $a_{n}<0$ for all $n\in\N$, 
$\mathcal{V}:=(\exp(a_{n}\mu))_{n\in\N}$ and $\Omega_{n}:=\{z\in\C\;|\;|\im(z)|<n\}$ for all $n\in\N$ 
where 
\[
 \mu\colon\C \to [0,\infty),\;\mu(z):=|\re(z)|^{\gamma},
\]
for some $0<\gamma\leq 1$. 
Then
\[
\overline{\partial}\colon \mathcal{EV}(\C)\to\mathcal{EV}(\C)
\]
is surjective.
\end{cor}

The restriction to negative $a_{n}$ comes from the condition that $-\ln\nu_{n}$ should be subharmonic. 
We note that the $E$-valued versions of 
\prettyref{thm:scalar_CR_surjective} and \prettyref{cor:scalar_surj_CR_example} where 
$E$ is a Fr\'echet space over $\C$ hold as well by the classical theory of tensor products for nuclear Fr\'echet spaces 
(see \cite[4.9 Corollary, p.\ 21]{kruse2018_5}). Since we will use the $\varepsilon$-product 
$\mathcal{EV}(\Omega)\varepsilon E$ to enlarge our collection of locally convex Hausdorff space $E$ 
for which $\overline{\partial}^{E}$ is surjective, we remark the following (cf.\ \cite[5.23 Lemma, p.\ 92]{ich}). 

\begin{prop}\label{prop:isom_op_spaces}
\begin{enumerate}
\item [a)] Let $X$ be a semi-reflexive locally convex Hausdorff space and $Y$ a Fr\'echet space. 
Then $L_{b}(X_{b}',Y_{b}')\cong L_{b}(Y,(X_{b}')_{b}')$ via taking adjoints.
\item [b)] Let $X$ be a Montel space and $E$ a locally convex Hausdorff space. 
Then $L_{b}(X_{b}',E)\cong X\varepsilon E$ where the topological isomorphism is the identity map.
\end{enumerate}
\end{prop}
\begin{proof}
$a)$ We consider the map
\[
{^{t}(\cdot)}\colon L_{b}(X_{b}',Y_{b}')\to L_{b}(Y,(X_{b}')_{b}'),\; u\mapsto {^{t}u},
\]
defined by ${^{t}u}(y)(x'):=u(x')(y)$ for $y\in Y$ and $x'\in X'$.
First, we prove that ${^{t}(\cdot)}$ is well-defined. Let $u\in L(X_{b}',Y_{b}')$ and $y\in Y$. 
Since $u\in L(X_{b}',Y_{b}')$ and 
$\{y\}$ is bounded in $Y$, there are a bounded set $B\subset X$ and $C>0$ such that
\[
|{^{t}u}(y)(x')|=|u(x')(y)|\leq C\sup_{x\in B}|x'(x)|
\]
for all $x'\in X'$ implying ${^{t}u}(y)\in(X_{b}')'$.
  
Let us denote by $(\|\cdot\|_{Y,n})_{n\in\N}$ the (directed) system of seminorms generating the metrisable locally convex topology of $Y$. 
The canonical embedding $J\colon Y\to (Y_{b}')_{b}'$ is a topological isomorphism between $Y$ and $J(Y)$ 
by \cite[Corollary 25.10, p.\ 298]{meisevogt1997} because $Y$ is a Fr\'echet space. 
For a bounded set $M\subset X_{b}'$ we note that
\[
 \sup_{x'\in M}|{^{t}u}(y)(x')|
=\sup_{x'\in M}|u(x')(y)|
=\sup_{x'\in M}|\langle J(y),u(x')\rangle|.
\]
The next step is to prove that $u(M)$ is bounded in $Y_{b}'$. 
Let $N\subset Y$ be bounded. Since $u\in L(X_{b}',Y_{b}')$, there is again a bounded set $B\subset X$ 
and a constant $C>0$ such that
\[
\sup_{x'\in M}\sup_{y\in N}|u(x')(y)|\leq C\sup_{x'\in M}\sup_{x\in B}|x'(x)|<\infty,
\]
where the last estimate follows from the boundedness of $M\subset X_{b}'$. 
Hence $u(M)$ is bounded in $Y_{b}'$.
By the remark about the canonical embedding there are $n\in\N$ and $C_{0}>0$ such that
\[
 \sup_{x'\in M}|{^{t}u}(y)(x')|
=\sup_{y'\in u(M)}|\langle J(y),y'\rangle|
\leq C_{0}\|y\|_{Y,n},
\]
so ${^{t}u}\in L(Y,(X_{b}')_{b}')$ and the map ${^{t}(\cdot)}$ is well-defined.

Let us turn to injectivity. Let $u,v\in L(X_{b}',Y_{b}')$ with ${^{t}u}={^{t}v}$. This is equivalent to
\[
u(x')(y)={^{t}u}(y)(x')={^{t}v}(y)(x')=v(x')(y)
\]
for all $y\in Y$ and $x'\in X'$. This implies $u(x')=v(x')$ for all $x'\in X'$, hence $u=v$.

Next, we turn to surjectivity. We consider the map
\[
{^{t}(\cdot)}\colon L_{b}(Y,(X_{b}')_{b}')\to L_{b}(X_{b}',Y_{b}'),\; u\mapsto {^{t}u},
\]
defined by ${^{t}u}(x')(y):=u(y)(x')$ for $x'\in X'$ and $y\in Y$. We show that this map is well-defined.
Let $u\in L_{b}(Y,(X_{b}')_{b}')$ and $x'\in X'$. 
Since $u\in L_{b}(Y,(X_{b}')_{b}')$ and $\{x'\}$ is bounded in $X'$, there are $n\in\N$ and $C>0$ such that
\[
|{^{t}u}(x')(y)|=|u(y)(x')|\leq C\|y\|_{Y,n}
\]
for all $y\in Y$ yielding to ${^{t}u}(x')\in Y'$.
Let $B\subset Y$ be bounded. The semi-reflexivity of $X$ implies that for every $u(y)$, $y\in B$, 
there is a unique $x_{u(y)}\in X$ such that $u(y)(x')=x'(x_{u(y)})$ for all $x'\in X'$. Then we get
\[
 \sup_{y\in B}|{^{t}u}(x')(y)|
=\sup_{y\in B}|u(y)(x')|
=\sup_{y\in B}|x'(x_{u(y)})|.
\]
We claim that $D:=\{x_{u(y)}\;|\;y\in B\}$ is a bounded set in $X$. Let $N\subset X'$ be finite. 
Then the set $M:=\{{^{t}u}(x')\;|\;x'\in N\}\subset Y'$ is finite. 
We have
\[
 \sup_{y\in B}\sup_{x'\in N}|x'(x_{u(y)})|
=\sup_{y\in B}\sup_{x'\in N}|{^{t}u}(x')(y)|
=\sup_{y\in B}\sup_{y'\in M}|y'(y)|<\infty
\]
where the last estimate follows from the fact that the bounded set $B$ is weakly bounded. 
Thus $D$ is weakly bounded and by \cite[Mackey's theorem 23.15, p.\ 268]{meisevogt1997} bounded in $X$. 
Therefore, it follows from
\[
 \sup_{y\in B}|{^{t}u}(x')(y)|
=\sup_{y\in B}|x'(x_{u(y)})|
=\sup_{x\in D}|x'(x)|
\]
for all $x'\in X'$ that ${^{t}u}\in L(X_{b}',Y_{b}')$ which means that ${^{t}(\cdot)}$ is well-defined.
Let $u\in L(Y,(X_{b}')_{b}')$. Then we have ${^{t}u}\in L_{b}(X_{b}',Y_{b}')$. 
In addition, for all $y\in Y$ and all $x'\in X'$
\[
{^{t}({^{t}u})}(y)(x')={^{t}u}(x')(y)=u(y)(x')
\]
is valid and so ${^{t}({^{t}u})}(y)=u(y)$ for all $y\in Y$ proving the surjectivity.

The last step is to prove the continuity of ${^{t}(\cdot)}$ and its inverse. 
Let $M\subset Y$ and $B\subset X_{b}'$ be bounded sets. Then
\begin{align*}
  \sup_{y\in M}\sup_{x'\in B}|{^{t}u}(y)(x')|
&=\sup_{y\in M}\sup_{x'\in B}|u(x')(y)|
 =\sup_{x'\in B}\sup_{y\in M}|u(x')(y)|\\
&=\sup_{x'\in B}\sup_{y\in M}|{^{t}({^{t}u})}(x')(y)|
\end{align*}
holds for all $u\in L(X_{b}',Y_{b}')$. Therefore, ${^{t}(\cdot)}$ and its inverse are continuous.
	 
$b)$ Let $T\in L(X_{b}',E)$. For $\alpha\in\mathfrak{A}$ there are a bounded set $B\subset X$ and $C>0$ such that
\[
p_{\alpha}(T(x'))\leq C\sup_{x\in B}|x'(x)|\leq C\sup_{x\in \oacx{(B)}}|x'(x)|
\]
for every $x'\in X'$. The set $\oacx{(B)}$ is absolutely convex and compact by \cite[6.2.1 Proposition, p.\ 103]{Jarchow} 
and \cite[6.7.1 Proposition, p.\ 112]{Jarchow} since $B$ is bounded in the Montel space $X$. 
Hence we gain $T\in L(X_{\kappa}',E)$.

Let $M\subset X'$ be equicontinuous. Due to \cite[8.5.1 Theorem (a), p.\ 156]{Jarchow} $M$ is bounded in $X_{b}'$. 
Therefore,
\[
\id\colon L_{b}(X_{b}',E)\to L_{e}(X_{\kappa}',E)=X\varepsilon E
\]
is continuous.

Let $T\in L(X_{\kappa}',E)$. For $\alpha\in\mathfrak{A}$ there are an absolutely convex compact set $B\subset X$ 
and $C>0$ such that
\[
p_{\alpha}(T(x'))\leq C\sup_{x\in B}|x'(x)|
\]
for every $x'\in X'$. Since the compact set $B$ is bounded, we get $T\in L(X_{b}',E)$.

Let $M$ be a bounded set in $X_{b}'$. Then $M$ is equicontinuous by virtue of \cite[Theorem 33.2, p.\ 349]{Treves}, 
as $X$, being a Montel space, is barrelled by \cite[Remark 24.24 (a), p.\ 286]{meisevogt1997}. 
Thus
\[
\id\colon L_{e}(X_{\kappa}',E)\to L_{b}(X_{b}',E)
\]
is continuous.
\end{proof}

Now, we use the results obtained so far and splitting theory to obtain our main theorem on the surjectivity of the 
Cauchy-Riemann operator on the space $\mathcal{EV}(\Omega,E)$. 
We recall that a Fr\'echet space $(F,(\norm{3}{\cdot}_{k})_{k\in\N})$ satisfies $(DN)$
by \cite[Chap.\ 29, Definition, p.\ 359]{meisevogt1997} if
\[
\exists\;p\in\N\;\forall\;k\in\N\;\exists\;n\in\N,\,C>0\;\forall\;x\in F:\;
\norm{3}{x}^{2}_{k}\leq C\norm{3}{x}_{p}\norm{3}{x}_{n}.
\]
A (PLS)-space is a projective limit $X=\lim\limits_{\substack{\longleftarrow\\N\in\N}}X_{N}$, where the
inductive limits $X_{N}=\lim\limits_{\substack{\longrightarrow\\n\in \N}}(X_{N,n},\norm{3}{\cdot}_{N,n})$ are (DFS)-spaces (which are also 
called (LS)-spaces), and it satisfies $(PA)$ if
\begin{gather*}
\forall \;N\;\exists\;  M\;  \forall\;  K\;  \exists\;  n\;  \forall\;  m\;  \forall\;  \eta >0\;  \exists\;  k,C,r_0 >0\;  \forall\;  r>r_0\; \forall\; x'\in X'_{N}:\\
	\norm{3}{x'\circ i^{M}_{N}}^{\ast}_{M,m}\leq C\bigl(r^{\eta}\norm{3}{x'\circ i^{K}_{N}}^{\ast}_{K,k}+\frac{1}{r}\norm{3}{x'}^{\ast}_{N,n}\bigr)
\end{gather*}
where $\norm{3}{\cdot}^{\ast}$ denotes the dual norm of $\norm{3}{\cdot}$ (see \cite[Section 4, Eq.\ (24), p.\ 577]{Dom1}).

\begin{thm}\label{thm:surj_CR_DN_PA}
Let \prettyref{cond:weights} with $\psi_{n}(z):=(1+|z|^{2})^{-2}$, $z\in\Omega$, and 
\prettyref{cond:dense} with $I_{214}(n)\geq I_{14}(n+1)$ be fulfilled and $-\ln\nu_{n}$ be subharmonic 
on $\Omega$ for every $n\in\N$. 
If $\mathcal{OV}(\Omega)$ satifies property $(\Omega)$ and 
\begin{enumerate}
\item [a)] $E:=F_{b}'$ where $F$ is a Fr\'echet space over $\C$ satisfying $(DN)$, or 
\item [b)] $E$ is an ultrabornological (PLS)-space over $\C$ satisfying $(PA)$, 
\end{enumerate}
then
\[
\overline{\partial}^{E}\colon \mathcal{EV}(\Omega,E)\to\mathcal{EV}(\Omega,E)
\]
is surjective.
\end{thm}
\begin{proof}
Throughout this proof we use the notation $X'':=(X_{b}')_{b}'$ for a locally convex Hausdorff space $X$.
In both cases, $a)$ and $b)$, the space $E$ is a complete locally convex Hausdorff space. 
The space $\mathcal{EV}(\Omega)$ is a Fr\'echet space by \cite[3.4 Proposition, p.\ 6]{kruse2018_2} 
and $\mathcal{OV}(\Omega)$ as well since it is a closed subspace by \prettyref{prop:equiv_seminorms} b). 
Both spaces are also nuclear and thus reflexive by \cite[3.1 Theorem, p.\ 12]{kruse2018_4}, 
\cite[2.7 Remark, p.\ 5]{kruse2018_4} and \cite[2.3 Remark b), p.\ 3]{kruse2018_4} 
because $(\omega.1)$ and $(\omega.2)^{1}$ from \prettyref{cond:weights} are fulfilled. 
As a consequence the map 
\[
 S\colon\mathcal{EV}(\Omega)\varepsilon E\to\mathcal{EV}(\Omega,E),\; 
 u\longmapsto [z\mapsto u(\delta_{z})],
\]
is a topological isomorphism by \cite[5.10 Example c), p.\ 24]{kruse2017} 
where $\delta_{z}$ is the point-evaluation at $z\in\Omega$. 
We denote by $\mathcal{J}\colon E\to E'^{\ast}$ the canonical injection in the algebraic dual $E'^{\ast}$ 
of the topological dual $E'$ and for $f\in\mathcal{EV}(\Omega,E)$ we set
\[
 R_{f}^{t}\colon\mathcal{EV}(\Omega)'\to E'^{\star},\;
 y\longmapsto\bigl[e'\mapsto y(e'\circ f) \bigr].
\]
Then the map $f\mapsto \mathcal{J}^{-1}\circ R_{f}^{t}$ is the inverse of $S$ 
by \cite[3.14 Theorem, p.\ 9]{kruse2017}.
The sequence
\begin{equation}\label{thm19.1}
0\to\mathcal{OV}(\Omega)\overset{i}{\to}\mathcal{EV}(\Omega)
\overset{\overline{\partial}}{\to}\mathcal{EV}(\Omega)\to 0,
\end{equation}
where $i$ means the inclusion, is an exact sequence of Fr\'echet spaces by \prettyref{thm:scalar_CR_surjective} 
and hence topologically exact as well. 
Let us denote by $J_{0}\colon\mathcal{OV}(\Omega)\to\mathcal{OV}(\Omega)''$ and 
$J_{1}\colon\mathcal{EV}(\Omega)\to\mathcal{EV}(\Omega)''$ the canonical embeddings 
which are topological isomorphisms since $\mathcal{OV}(\Omega)$ and
$\mathcal{EV}(\Omega)$ are reflexive. 
Then the exactness of \eqref{thm19.1} implies that
\begin{equation}\label{thm19.2}
0\to\mathcal{OV}(\Omega)''\overset{i_{0}}{\to}\mathcal{EV}(\Omega)''
\overset{\overline{\partial}_{1}}{\to}\mathcal{EV}(\Omega)''\to 0,
\end{equation}
where $i_{0}:=J_{0}\circ i\circ J^{-1}_{0}$ and $\overline{\partial}_{1}:=J_{1}\circ \overline{\partial}\circ J^{-1}_{1}$, 
is an exact topological sequence. 
Topological as the (strong) bidual of a Fr\'echet space is again a Fr\'echet space 
by \cite[Corollary 25.10, p.\ 298]{meisevogt1997}.

$a)$ Let $E:=F_{b}'$ where $F$ is a Fr\'echet space with $(DN)$. 
Then $\operatorname{Ext}^{1}(F,\mathcal{OV}(\Omega)'')=0$ by \cite[5.1 Theorem, p.\ 186]{V1} 
since $\mathcal{OV}(\Omega)$ satisfies $(\Omega)$ and therefore $\mathcal{OV}(\Omega)''$ as well. 
Combined with the exactness of \eqref{thm19.2} this implies that the sequence
\[
0\to L(F,\mathcal{OV}(\Omega)'')\overset{i^{\ast}_{0}}{\to}L(F,\mathcal{EV}(\Omega)'')
\overset{\overline{\partial}^{\ast}_{1}}{\to}L(F,\mathcal{EV}(\Omega)'')\to 0
\]
is exact by \cite[Proposition 2.1, p.\ 13-14]{Pala} where $i^{\ast}_{0}(B):=i_{0}\circ B$ and
$\overline{\partial}^{\ast}_{1}(D):=\overline{\partial}_{1}\circ D$ for 
$B\in L(F,\mathcal{OV}(\Omega)'')$ and $D\in L(F,\mathcal{EV}(\Omega)'')$. 
In particular, we obtain that
\begin{equation}\label{thm19.3}
\overline{\partial}^{\ast}_{1}\colon L(F,\mathcal{EV}(\Omega)'')\to
L(F,\mathcal{EV}(\Omega)'')
\end{equation}
is surjective. 
Via $E=F_{b}'$ and \prettyref{prop:isom_op_spaces} ($X=\mathcal{EV}(\Omega)$ and $Y=F$) we have 
the topological isomorphism
\[
\psi:=S\circ{^{t}(\cdot)}\colon L(F,\mathcal{EV}(\Omega)'')\to 
\mathcal{EV}(\Omega,E),\; 
\psi(u)=\bigl(S\circ{^{t}(\cdot)}\bigr)(u)=\bigl[z\mapsto{^{t}u}(\delta_{z})\bigr],
\]
and the inverse
\[ 
 \psi^{-1}(f)
=(S\circ{^{t}(\cdot)})^{-1}(f)
=({^{t}(\cdot)}\circ S^{-1})(f)
={^{t}(\mathcal{J}^{-1}\circ R_{f}^{t})},\quad f\in\mathcal{EV}(\Omega,E).
\]
Let $g\in\mathcal{EV}(\Omega,E)$. 
Then $\psi^{-1}(g)\in L(F,\mathcal{EV}(\Omega)'')$ and by the surjectivity of \eqref{thm19.3} 
there is $u\in L(F,\mathcal{EV}(\Omega)'')$ such that 
$\overline{\partial}^{\ast}_{1}u=\psi^{-1}(g)$. 
So we get $\psi(u)\in\mathcal{EV}(\Omega,E)$. 
Next, we show that $\overline{\partial}^{E}\psi(u)=g$ is valid. 
Let $x\in F$, $z\in\Omega$ and $h\in\R$, $h\neq 0$, 
and $e_{k}$ denote the $k$th unit vector in $\R^{2}$. 
From 
\[
\bigl(\frac{\delta_{z+he_{k}}-\delta_{z}}{h}\bigr)(f)=\frac{f(z+he_{k})-f(z)}{h}\underset{h\to 0}{\to}\partial^{e_{k}}f(z),
\]
for every $f\in\mathcal{EV}(\Omega)$ follows that $\frac{\delta_{z+he_{k}}-\delta_{z}}{h}$ 
converges to $\delta_{z}\circ\partial^{e_{k}}$ in $\mathcal{EV}(\Omega)_{\sigma}'$. 
Since the nuclear Fr\'echet space $\mathcal{EV}(\Omega)$ is in particular a Montel space, we deduce 
that $\frac{\delta_{z+he_{k}}-\delta_{z}}{h}$ converges to $\delta_{z}\circ\partial^{e_{k}}$ in 
$\mathcal{EV}(\Omega)_{\gamma}'=\mathcal{EV}(\Omega)_{b}'$ 
by the Banach-Steinhaus theorem. Let $B\subset F$ be bounded. 
As ${^{t}u}\in L(\mathcal{EV}(\Omega)_{b}',F_{b}')$, there are a bounded set 
$B_{0}\subset\mathcal{EV}(\Omega)$ and $C>0$ such that 
\begin{flalign*}
&\hspace{0.37cm}\sup_{x\in B}\bigl|\bigl(\frac{{^{t}u}(\delta_{z+he_{k}})-{^{t}u}(\delta_{z})}{h}\bigr)(x)
                                   -{^{t}u}\bigl(\delta_{z}\circ\partial^{e_{k}}\bigr)(x)\bigr|\\
&=\sup_{x\in B}\bigl|{^{t}u}\bigl(\frac{\delta_{z+he_{k}}-\delta_{z}}{h}-\delta_{z}\circ\partial^{e_{k}}\bigr)(x)\bigr|
 \leq C\sup_{f\in B_{0}}\bigl|\bigl(\frac{\delta_{z+he_{k}}-\delta_{z}}{h}-\delta_{z}\circ\partial^{e_{k}}\bigr)(f)\bigr|
\underset{h\to 0}{\to}0
\end{flalign*}
yielding to $(\partial^{e_{k}})^{E}(\psi(u))(z)={^{t}u}(\delta_{z}\circ\partial^{e_{k}})$. 
This implies $\overline{\partial}^{E}(\psi(u))(z)={^{t}u}(\delta_{z}\circ\overline{\partial})$. 
So for all $x\in F$ and $z\in\Omega$ we have
\begin{align*}
 \overline{\partial}^{E}(\psi(u))(z)(x)
&={^{t}u}(\delta_{z}\circ\overline{\partial})(x)
 =u(x)(\delta_{z}\circ\overline{\partial})
 =\langle\delta_{z}\circ\overline{\partial},J^{-1}_{1}(u(x))\rangle\\
&=\langle\delta_{z},\overline{\partial}J^{-1}_{1}(u(x))\rangle
 =\langle[J_{1}\circ\overline{\partial}\circ J^{-1}_{1}](u(x)),\delta_{z}\rangle
 =\langle(\overline{\partial}_{1}\circ u)(x),\delta_{z}\rangle\\
&=\langle(\overline{\partial}^{\ast}_{1}u)(x),\delta_{z}\rangle
 =\psi^{-1}(g)(x)(\delta_{z})
 ={^{t}(\mathcal{J}^{-1}\circ R_{g}^{t})}(x)(\delta_{z})\\
&=(\mathcal{J}^{-1}\circ R_{g}^{t})(\delta_{z})(x)
 =\mathcal{J}^{-1}(\mathcal{J}(g(z))(x)
 =g(z)(x).
\end{align*}
Thus $\overline{\partial}^{E}(\psi(u))(z)=g(z)$ for every $z\in\Omega$ which proves the surjectivity.

$b)$ Let $E$ be an ultrabornological (PLS)-space satisfying $(PA)$. 
Since the nuclear Fr\'echet space $\mathcal{OV}(\Omega)$ is also a Schwartz space, 
its strong dual $\mathcal{OV}(\Omega)_{b}'$ is a (DFS)-space. 
By \cite[Theorem 4.1, p.\ 577]{Dom1} we obtain $\operatorname{Ext}^{1}_{PLS}(\mathcal{OV}(\Omega)_{b}',E)=0$ 
as the bidual $\mathcal{OV}(\Omega)''$ satisfies $(\Omega)$, $E$ is a (PLS)-space satisfying $(PA)$ 
and condition (c) in the theorem is fulfilled because $\mathcal{OV}(\Omega)_{b}'$ is the strong dual of 
a nuclear Fr\'echet space.
Moreover, we have $\operatorname{Proj}^{1} E=0$ due to \cite[Corollary 3.3.10, p.\ 46]{Wengen} 
because $E$ is an ultrabornological (PLS)-space. 
Then the exactness of the sequence \eqref{thm19.2}, \cite[Theorem 3.4, p.\ 567]{Dom1} and \cite[Lemma 3.3, p.\ 567]{Dom1} 
(in the lemma the same condition (c) as in \cite[Theorem 4.1, p.\ 577]{Dom1} is fulfilled 
and we choose $H=\mathcal{OV}(\Omega)''$ and 
$F=G=\mathcal{EV}(\Omega)''$), imply that the sequence
\[
0\to L(E_{b}',\mathcal{OV}(\Omega)'')
\overset{i^{\ast}_{0}}{\to}L(E_{b}',\mathcal{EV}(\Omega)'')
\overset{\overline{\partial}^{\ast}_{1}}{\to}L(E_{b}',\mathcal{EV}(\Omega)'')\to 0
\]
is exact. The maps $i^{\ast}_{0}$ and $\overline{\partial}^{\ast}_{1}$ are defined like in part $a)$. 
Especially, we get that
\begin{equation}\label{thm19.4}
\overline{\partial}^{\ast}_{1}\colon L(E_{b}',\mathcal{EV}(\Omega)'')\to
L(E_{b}',\mathcal{EV}(\Omega)'')
\end{equation}
is surjective. 

By \cite[Remark 4.4, p.\ 1114]{D/L} we have $L_{b}(\mathcal{EV}(\Omega)_{b}',E'')
\cong L_{b}(E_{b}',\mathcal{EV}(\Omega)'')$ via taking adjoints 
since $\mathcal{EV}(\Omega)$, being a Fr\'echet-Schwartz space, is a (PLS)-space and hence 
its strong dual an (LFS)-space, which is regular by \cite[Corollary 6.7, $10.\Leftrightarrow 11.$, p.\ 114]{Wengen}, 
and $E$ is an ultrabornological (PLS)-space, in particular, reflexive by \cite[Theorem 3.2, p.\ 58]{Dom2}. 
In addition, the map
\[
T\colon L_{b}(\mathcal{EV}(\Omega)_{b}',E'')\to 
L_{b}(\mathcal{EV}(\Omega)_{b}',E),
\]
defined by $T(u)(y):=\mathcal{J}^{-1}(u(y))$ for $u\in L(\mathcal{EV}(\Omega)_{b}',E'')$ 
and $y\in \mathcal{EV}(\Omega)'$, is a topological isomorphism because $E$ is reflexive.
Due to \prettyref{prop:isom_op_spaces} b) we obtain the topological isomorphism
\begin{gather*}
\psi:=S\circ\mathcal{J}^{-1}\circ{^{t}(\cdot)}\colon L_{b}(E_{b}',\mathcal{EV}(\Omega)'')
\to\mathcal{EV}(\Omega,E),\\
\psi(u)=[S\circ\mathcal{J}^{-1}\circ{^{t}(\cdot)}](u)=\bigl[z\mapsto\mathcal{J}^{-1}({^{t}u}(\delta_{z}))\bigr],
\end{gather*}
with the inverse given by 
\[ 
 \psi^{-1}(f)
=(S\circ\mathcal{J}^{-1}\circ{^{t}(\cdot)})^{-1}(f)
=[{^{t}(\cdot)}\circ\mathcal{J}\circ S^{-1}](f)
={^{t}(\mathcal{J}\circ\mathcal{J}^{-1}\circ R_{f}^{t})}={^{t}(R_{f}^{t})}
\]
for $f\in\mathcal{EV}(\Omega,E)$.

Let $g\in\mathcal{EV}(\Omega,E)$. 
Then $\psi^{-1}(g)\in L(E_{b}',\mathcal{EV}(\Omega)'')$ 
and by the surjectivity of \eqref{thm19.4} there exists $u\in L(E_{b}',\mathcal{EV}(\Omega)'')$ 
such that $\overline{\partial}^{\ast}_{1}u=\psi^{-1}(g)$. 
So we have $\psi(u)\in\mathcal{EV}(\Omega,E)$. 
The last step is to show that $\overline{\partial}^{E}\psi(u)=g$. 
Like in part a) we gain for every $z\in\Omega$
\[	
\overline{\partial}^{E}(\psi(u))(z)=\mathcal{J}^{-1}({^{t}u}(\delta_{z}\circ\overline{\partial}))
\]
and for every $x\in E'$
\begin{align*}
  {^{t}u}(\delta_{z}\circ\overline{\partial})(x)
&=u(x)(\delta_{z}\circ\overline{\partial})
 =(\overline{\partial}^{\ast}_{1}u)(x)(\delta_{z})
 =\psi^{-1}(g)(x)(\delta_{z})
 ={^{t}(R_{g}^{t})}(x)(\delta_{z})\\
&=\delta_{z}(x\circ g)
 =x(g(z))
 =\mathcal{J}(g(z))(x).
\end{align*}
Thus we have ${^{t}u}(\delta_{z}\circ\overline{\partial})=\mathcal{J}(g(z))$ and therefore
$\overline{\partial}^{E}(\psi(u))(z)=g(z)$ for all $z\in\Omega$.
\end{proof}

Due to \cite[1.4 Lemma, p.\ 110]{Vogt1977} and \cite[Proposition 4.2, p.\ 577]{Dom1} we have 
the following relation between the cases $a)$ and $b)$ in \prettyref{thm:surj_CR_DN_PA}.

\begin{rem}\label{rem:link_PA_DN}
Let $F$ be a Fr\'echet-Schwartz space. 
Then $F$ satisfies $(DN)$ if and only if the (DFS)-space $E:=F_{b}'$ satisfies $(PA)$. 
\end{rem}

Thus case $a)$ is included in case $b)$ if $F$ is a Fr\'echet-Schwartz space. Therefore $a)$ is only 
interesting for Fr\'echet spaces $F$ which are not Schwartz spaces. 

\begin{cor}\label{cor:surj_CR_DN_PA}
Let $\mu$ be a subharmonic strong weight generator, $(a_{n})_{n\in\N}$ strictly increasing, $a_{n}< 0$ for all $n\in\N$, 
$\lim_{n\to\infty}a_{n}=0$ and $\mathcal{V}:=(\exp(a_{n}\mu))_{n\in\N}$.
Let \prettyref{cond:weights} with $\psi_{n}(z):=(1+|z|^{2})^{-2}$, $z\in\C$, and 
\prettyref{cond:dense} with $I_{214}(n)\geq I_{14}(n+1)$ and $\Omega_{n}:=\{z\in\C\;|\;|\im(z)|<n\}$ 
for all $n\in\N$ be fulfilled.
If 
\begin{enumerate}
\item [a)] $E:=F_{b}'$ where $F$ is a Fr\'echet space over $\C$ satisfying $(DN)$, or 
\item [b)] $E$ is an ultrabornological (PLS)-space over $\C$ satisfying $(PA)$, 
\end{enumerate}
then
\[
\overline{\partial}^{E}\colon \mathcal{EV}(\C,E)\to\mathcal{EV}(\C,E)
\]
is surjective.
\end{cor}
\begin{proof}
The assertion is a direct consequence of \prettyref{thm:surj_CR_DN_PA} and 
\prettyref{thm:omega_for_strips}. 
\end{proof}

\prettyref{cor:surj_CR_DN_PA} generalises a part of \cite[5.24 Theorem, p.\ 95]{ich} ($K=\varnothing$) 
which is the case $\gamma=1$ of the next corollary. 

\begin{cor}\label{cor:surj_CR_DN_PA_exa}
Let $(a_{n})_{n\in\N}$ be strictly increasing, $a_{n}<0$ for all $n\in\N$, $\lim_{n\to\infty}a_{n}=0$, 
$\mathcal{V}:=(\exp(a_{n}\mu))_{n\in\N}$ and $\Omega_{n}:=\{z\in\C\;|\;|\im(z)|<n\}$ for all $n\in\N$ 
where 
\[
 \mu\colon\C \to [0,\infty),\;\mu(z):=|\re(z)|^{\gamma},
\]
for some $0<\gamma\leq 1$. If 
\begin{enumerate}
\item [a)] $E:=F_{b}'$ where $F$ is a Fr\'echet space over $\C$ satisfying $(DN)$, or 
\item [b)] $E$ is an ultrabornological (PLS)-space over $\C$ satisfying $(PA)$, 
\end{enumerate}
then
\[
\overline{\partial}^{E}\colon \mathcal{EV}(\C,E)\to\mathcal{EV}(\C,E)
\]
is surjective.
\end{cor}
\begin{proof}
Follows from \prettyref{cor:surj_CR_DN_PA} and \prettyref{cor:omega_for_strips}. 
\end{proof}

To close this section we provide some examples of ultrabornological (PLS)-spaces satisfying $(PA)$ and 
spaces of the form $E:=F_{b}'$ where $F$ is a Fr\'echet space satisfying $(DN)$.

\begin{exa}\label{exa:PLS_PA_DF_DN}
a) The following spaces are ultrabornological (PLS)-spaces with property $(PA)$ 
and also strong duals of a Fr\'echet space satisfying $(DN)$:
\begin{itemize}
 \item the strong dual of a power series space of inifinite type $\Lambda_{\infty}(\alpha)_{b}'$,
 \item the strong dual of any space of holomorphic functions $\mathcal{O}(U)_{b}'$ 
 where $U$ is a Stein manifold with the strong Liouville property (for instance, for $U=\C^{d}$),
 \item the space of germs of holomorphic functions $\mathcal{O}(K)$ where $K$ is a completely pluripolar 
 compact subset of a Stein manifold (for instance $K$ consists of one point),
 \item the space of tempered distributions $\mathcal{S}(\R^{d})_{b}'$ and 
 the space of Fourier ultra-hyperfunctions $\mathcal{P}'_{\ast\ast}$ (with the strong topology),
 \item the weighted distribution spaces $(K\{pM\})_{b}'$ of Gelfand and Shilov if the weight $M$ satisfies
 \[
  \sup_{|y|\leq 1}M(x+y)\leq C\inf_{|y|\leq 1}M(x+y),\quad x\in\R^{d},
 \]
 \item $\mathcal{D}(K)_{b}'$ for any compact set $K\subset\R^{d}$ with non-empty interior,
 \item $\mathcal{C}^{\infty}(\overline{U})_{b}'$ for any non-empty open bounded set $U\subset\R^{d}$ with $\mathcal{C}^{1}$-boundary.
\end{itemize}
b) The following spaces are ultrabornological (PLS)-spaces with property $(PA)$:
\begin{itemize}
 \item an arbitrary Fr\'echet-Schwartz space,
 \item a (PLS)-type power series space $\Lambda_{r,s}(\alpha,\beta)$ whenever $s=\infty$ 
 or $\Lambda_{r,s}(\alpha,\beta)$ is a Fr\'echet space,
 \item the spaces of distributions $\mathcal{D}(U)_{b}'$ and ultradistributions of Beurling type $\mathcal{D}_{(\omega)}(U)_{b}'$ 
 for any open set $U\subset\R^{d}$,
 \item the kernel of any linear partial differential operator with constant coefficients in $\mathcal{D}(U)_{b}'$ 
 or in $\mathcal{D}_{(\omega)}(U)_{b}'$ when $U\subset\R^{d}$ is open and convex,
 \item the space $L_{b}(X,Y)$ where $X$ has $(DN)$, $Y$ has $(\Omega)$ and both are nuclear Fr\'echet spaces. 
 In particular, $L_{b}(\Lambda_{\infty}(\alpha),\Lambda_{\infty}(\beta))$ if both spaces are nuclear.
\end{itemize}
c) The following spaces are strong duals of a Fr\'echet space satisfying $(DN)$:
\begin{itemize}
 \item the strong dual $F_{b}'$ of any Banach space $F$,
 \item the strong dual $\lambda^{2}(A)_{b}'$ of the K\"othe space $\lambda^{2}(A)$ 
 with a K\"othe matrix $A=(a_{j,k})_{j,k\in\N_{0}}$ satisfying 
 \[
  \exists\;p\in\N_{0}\;\forall\;k\in\N_{0}\;\exists\;n\in\N_{0},C>0:\;a_{j,k}^{2}\leq Ca_{j,p}a_{j,n}.
 \]
\end{itemize}
\end{exa} 
\begin{proof}
The statement for the spaces in a) and b) follows from \cite[Corollary 4.8, p.\ 1116]{D/L}, 
\cite[Proposition 31.12, p.\ 401]{meisevogt1997}, 
\cite[Proposition 31.16, p.\ 402]{meisevogt1997} and \prettyref{rem:link_PA_DN}. 
The first part of statement c) is obvious since Banach spaces clearly satisfy the property $(DN)$. 
The second part on the K\"othe space $\lambda^{2}(A)$ follows from \cite[Satz 12.11 a), p.\ 305]{Kaballo}.
\end{proof}

We note that the cases that $E$ is a Fr\'echet-Schwartz space or that $E=\Lambda_{r,s}(\alpha,\beta)$ 
is a Fr\'echet space or that $E=F_{b}'$ where $F$ is a Banach space are already contained in the case 
that $E$ is a Fr\'echet space (see \cite[4.9 Corollary, p.\ 21]{kruse2018_5}).

\subsection*{Acknowledgements}
The present paper is a generalisation of parts of Chapter 5 of the author's Ph.D Thesis \cite{ich}, 
written under the supervision of M.\ Langenbruch. The author is deeply grateful to him for his support and advice. 
Further, it is worth to mention that some of the results appearing in the Ph.D Thesis 
and thus their generalised counterparts in this work are essentially due to him.
\bibliography{biblio}
\bibliographystyle{plain}
\end{document}